\documentclass[11pt,english]{article} 
\usepackage{babel}

\usepackage[latin9]{inputenc} 

\usepackage{geometry} 
\usepackage{graphicx} 

\usepackage{booktabs} 
\usepackage{array} 
\usepackage{paralist} 
\usepackage{verbatim} 
\usepackage{subfig} 

\usepackage{amsthm}
\usepackage{amsmath}
\usepackage{amssymb}
\usepackage{graphicx}
\usepackage{setspace}
\usepackage{esint}
\usepackage{epstopdf}
\usepackage{color}

\def\R{{\mathbb R}}
\def\N{{\mathbb N}}

\def\le{\leqslant}
\def\ge{\geqslant}

\newcommand{\eps}{\varepsilon}

\theoremstyle{plain}
\newtheorem{theorem}{Theorem}[section]

\theoremstyle{definition}

\newtheorem*{remark*}{Remark}

\numberwithin{equation}{section}

\title{Positivity-preserving and asymptotic preserving method for 2D Keller-Segal equations}

\author{Jian-Guo Liu\thanks{Department of Mathematics and Department of Physics, Duke University, Box 90320, Durham NC 27708, USA (jliu@phy.duke.edu)}, Li Wang\thanks{Department of Mathematics and Computational and Data-Enabled Science and Engineering Program, SUNY at Buffalo, 244 Mathematics Building, Buffalo, NY 14260, USA (lwang46@buffalo.edu)}, and Zhennan Zhou\thanks{Department of Mathematics, Duke University, Box 90320, Durham NC 27708, USA (zhennan@math.duke.edu)} }

\begin{document}
\maketitle

\begin{abstract}
We propose a semi-discrete scheme for 2D Keller-Segel equations based on a symmetrization reformation, which is equivalent to the convex splitting method and is free of any nonlinear solver. We show that, this new scheme is stable as long as the initial condition does not exceed certain threshold, and it asymptotically preserves the quasi-static limit in the transient regime. Furthermore, we prove that the fully discrete scheme is conservative and positivity preserving, which makes it ideal for simulations. The analogical schemes for the radial symmetric cases and the subcritical degenerate cases are also presented and analyzed. With extensive numerical tests, we verify the claimed properties of the methods and demonstrate their superiority in various challenging applications.

\end{abstract}

\section{Introduction}

In this paper, we consider the following 2D Keller-Segel equations 
\begin{align}
\partial_t \rho^\eps &= \Delta \rho^\eps - \nabla \cdot (\rho^\eps \nabla c^\eps), \quad x\in \R^2, t>0 \label{eq:rho0}\\
\eps \partial_t c^\eps & =\Delta c^\eps  + \rho^\eps, \quad x\in \R^2, t>0 \label{eq:c0}\\
\rho^{\eps}(x,0)& =f(x),\quad c^{\eps}(x,0)=g(x).
\end{align}
This system was originally established by Patlak \cite{Patlak1953} and Keller \& Segel \cite{KellerSegel1971} to model the phenomenon of chomotaxis, in which cells approach the chemically favorable environments according to the chemical substance generated by cells . Here $\rho^\eps(x,t)$ denotes the density distribution {of} cells and $c^{\eps}(x,t)$ denotes { the} chemical concentration. Mathematically, this model describes the
competition between the diffusion and the nonlocal aggregation. This type of competition is ubiquitous in evolutionary systems arisen in biology, social science and other interacting particle systems, numerous mathematical studies of the Keller-Segel system and its variants have been conducted in recent years; see \cite{Perthame} for a general discussion. 

When $\eps>0$, the system (\ref{eq:rho0}) (\ref{eq:c0}) is called the parabolic-parabolic model, whereas when $\eps=0$, it is called the parabolic-elliptic model.  When $\eps \ll 1$, the model is in a transition regime between the parabolic-parabolic and the parabolic-elliptic cases. For the parabolic-elliptic model, it is well known that $M_c = 8\pi$ is the critical mass that distinguishes the global-existent solution from finite-time blow up solution by utilizing the logarithmic Hardy-Littlewood Sobolev inequality \cite{BDP2006,Perthame}. More recently, Liu and Wang have proved the uniqueness of the weak solutions when the initial mass is less than $8 \pi$ and the initial free energy and the second moment are finite \cite{LiuWang2016}. For the parabolic-parabolic model, the global existence is analyzed and the critical mass (which is also $8\pi$) is derived in \cite{CalvezCorrias2008}. Most analytical results rely on the variational formation. 

In particular, we denote the free energy of the parabolic-parabolic system as
\begin{equation} \label{eqn:energy1}
\mathcal F (\rho,c)=\int_{\mathbb R^2}\left[  \rho \log \rho -\rho -\rho c + \frac{1}{2}|\nabla c|^2\right] dx,
\end{equation}
where we have suppressed the superscript $\eps$ for simplicity; see \cite{BCKKLL2014,CongLiu}. Then the system \eqref{eq:rho0} and \eqref{eq:c0} can be formulated by the following mixed conservative and nonconservative gradient flow
\[
\rho_t = \nabla \cdot \left(\rho \nabla \frac{\delta \mathcal F}{\delta \rho} \right), \quad c_t = - \frac{\delta \mathcal F}{\delta c}.
\]
This mixed variational structure is known as the Le Ch\"aterlier Principle. Formally when $\rho$ and $c$ solve the parabolic-parabolic system, the free energy $\mathcal F (t) =\mathcal F (\rho(\cdot,t),c(\cdot,t)) $ satisfies the following entropy-dissipation equality
\[
\frac{d}{dt} \mathcal F (t) + \int_{\R^2} \left[ \rho \left|\nabla \left(\log \rho -c \right) \right|^2 + |\partial_t c|^2 \right]dx =0.
\] 
In the parabolic-elliptic case, one can replace the equation of $c$ using the Newtonian potential
\[
c(x,t)= \frac{1}{2\pi} \log |x| * \rho (x,t),
\]
and the free energy for some proper $\rho$ is given by
\begin{equation} \label{eq:111}
\mathcal F (\rho)= \int_{\mathbb R^2}\left[  \rho \log \rho -\rho \right]dx + \frac{1}{2} \int_{\R^2 \times \R^2} \frac{1}{2\pi} \log |x-y| \rho(x) \rho(y) dx\, dy.
\end{equation}

We also consider the extension of the 2D Keller-Segel equations with degenerate diffusion
\begin{align}
\partial_t \rho^\eps &= \Delta (\rho^\eps)^m - \nabla \cdot (\rho^\eps \nabla c^\eps), \quad x\in \R^2, t>0 \label{eq:rhom0} \\
\eps \partial_t c^\eps & =\Delta c^\eps  + \rho^\eps, \quad x\in \R^2, t>0 \label{eq:cm0}\\
\rho^{\eps}(x,0)& =f(x),\quad c^{\eps}(x,0)=g(x).
\end{align}
Here $m$ is the diffusion exponent, and we call it supercritical when $0<m<1$, critical when $m=1$ and subcritical when $m>1$.  It is worth noting that the classification of the exponent is dimension dependent, the readers may refer to \cite{BianLiu2013,BrennerCKSV} for a broad summary.  {The free energy can be similarly defined  for this system and the entropy-dissipation equality can be derived , which we shall skip in this paper.}


While the Keller-Segel equations have been well studied and understood in the analytical aspect, there is much to explore in the numerical computations. Owing to the similarity to the drift-diffusion equation, Filbet proposed an implicit Finite Volume Method (FVM) for the Keller-Segel model \cite{Filbet2006}. However, instead of being repulsive,  the aggregation term in the Keller-Segel equation is attractive which competes  against the diffusion term, the FVM  method is constrained by severe stability constraint. In \cite{ChertockKurganov2008}, Chertock and Kurganov designed a second-order positivity preserving central-upwind
scheme for the chemotaxis models by converting the Keller-Segel equations to an advection-reaction-diffusion system. The main issue there is that the Jacobian matrices coming from the advection part may have complex eigenvalues, which force the advection part to be solved together with the stabilizing diffusion terms, and result in complicated CFL conditions. Based on this formulation, Kurganov and his collaborators have conducted many extensions, including more general chemotaxis flux model, multi-species model and constructing an alternative discontinuous Galerkin method; see \cite{ChertockKurganovWangWu2012,EpshteynKurganov2008,KurganovLM2014}.  Very recently, Li et. al have improved the results in by introducing the local discontinuous Galerkin method with optimal rate of convergence \cite{LiShuYang}. Another drawback of the methods based on the  advection-reaction-diffusion formulation is, in the transient regime when $\eps \ll 1$,  this methods suffer from the stiffness in $\eps$ and the stability constrains are therefore magnified. Besides, there is a kinetic formulation modeling the competition of diffusion and nonlocal aggregation, and some works on numerical simulation are available in \cite{CarrilloYan2013, ChengGamba2012}.
  
In this work, we aim to develop a numerical method which preserves both positivity and asymptotic limit. Namely, the numerical method does not generate negative density if initialized properly under a less strict stability condition. Moreover, such condition does not deteriorate with the decreasing of $\eps$, and when $\eps \rightarrow 0$, the discrete scheme of the parabolic-parabolic system automatically becomes a stable solver to the parabolic-elliptic system. In other words, we expect the numerical method to preserve the quasi-static limit of the Keller-Segel system in the transient regime. 

The key ingredient in our scheme is the following reformulation of the density equation \eqref{eq:rho0} 
\begin{equation}\label{eq:sym1}
\partial_t \rho^\eps  = \nabla \cdot \left( e^{c^\eps} \nabla \left( \frac{\rho^\eps}{e^{c^\eps}} \right) \right),
\end{equation}
which is reminiscent of the symmetric Fokker-Planck equation. Therefore, we can propose a semi-discrete approximation of (\ref{eq:sym1}) in the following way
\begin{equation}\label{semi1}
\frac{\rho^{n+1}-\rho^n}{\Delta t}=\nabla \cdot \left( e^{c(\rho^n)} \nabla \left( \frac{\rho^{n+1}}{e^{c(\rho^n)}} \right) \right).
\end{equation}
It is interesting to point out that the above discretization (\ref{semi1}) is equivalent to a first order convex splitting scheme \cite{GLWS2014}. To see this,  we reformulate \eqref{semi1} as
\[
\frac{\rho^{n+1}-\rho^n}{\Delta t} = \Delta \rho^{n+1}-\nabla \cdot \left(\rho^{n+1} \nabla c(\rho^n) \right)= \nabla \cdot \left(\rho^{n+1} \nabla \log \rho^{n+1} \right)-\nabla \cdot \left(\rho^{n+1} \nabla c(\rho^n) \right).
 \] 
{Further, we use the finite difference approximation to the spatial discretization.} The analog of the equation with the diffusion exponent $m\ne 1$ is 
\begin{equation}\label{eq:symn1}
\partial_t \rho^\eps  = \nabla \cdot \left[ \rho^\eps \exp\left(c^\eps-\frac{m}{m-1}(\rho^\eps)^{m-1}\right) \nabla \exp\left(-c^\eps+\frac{m}{m-1}(\rho^\eps)^{m-1} \right) \right].
\end{equation}
We shall design numerical methods based on this formulation as well.

The rest of the paper is organized as follows. We conduct asymptotic analysis to the Keller-Segel equations in the transient regime ($\eps \ll 1$) in Section \ref{sec:asymp}. In Section \ref{sec:m1}, we give a detailed construction and analysis of the numerical method, prove its stability, asymptotic preserving and positivity preserving properties, explore its high order accuracy analog and discuss its simplified structure in radial symmetric cases. In Section \ref{sec:mn1}, we extend the numerical method to the Keller-Segel equations with degenerate diffusions. Several numerical examples are given in the last section to verify the claimed properties and demonstrate its application in various challenging cases, including blow-up solutions, degenerate diffusions with large $m$ (see \cite{CraigKimYao}) and two-species models with different blowup behavior (see \cite{KurganovLM2014}). 


\subsection{Asymptotic analysis for the quasi-static limit}  \label{sec:asymp}

We carry out the asymptotic analysis to the solutions of the Keller-Segel equations (\ref{eq:rho0}) (\ref{eq:c0}) when $\eps \ll 1$ in the following. Due to the presence of the small parameter $\eps$, the solution $c^\eps$ is expected to experience a transient layer with a fast time scale $\tau=t/\eps$. In particular, we construct the following ansatz for solutions
\[
\rho^\eps(x,t)=\rho_\eps^0(x,t)+\eps \rho_\eps^1(x,t) \, ;
\]
\[
c^\eps(x,t)=c^0_{\eps,\text{in}}(x, \tau)+c^0_{\eps,\text{out}}(x,t)+\eps c^1_\eps (x,t) \, ,
\]
where $c_{\text{in}}(x,\tau)$ represents the solution inside the transition layer and thus depends on $\tau$. Plugging this ansatz into the equations  (\ref{eq:rho0}) (\ref{eq:c0}) and collecting the systems due to their orders, we have, to the leading order: 
\begin{align}
\partial_t \rho^0_\eps & = \Delta \rho^0_\eps + \nabla \cdot \left(\rho^0_\eps \nabla \left( c^0_{\eps,\text{in}}+c^0_{\eps,\text{out}} \right)  \right), \\
\partial_\tau c^0_{\eps,\text{in}} &= \Delta c^0_{\eps,\text{in}}, \\
0 & = \Delta  c^0_{\eps,\text{out}} + \rho^0_\eps. \label{eq:cout} 
\end{align}
The initial conditions are given by
\begin{align}
\rho^0_\eps(x,0)=f(x), \quad
c^0_{\eps,\text{in}}(x,0)+c^0_{\eps,\text{out}}(x,0) = g(x).\label{init:c0}
\end{align}
Clearly, equations \eqref{eq:cout}--\eqref{init:c0} imply that
\[
c^0_{\eps,\text{out}}(x,0) = (-\Delta)^{-1} f(x), \quad  c^0_{\eps,\text{in}}(x,0)=g(x)-(-\Delta)^{-1} f(x).
\]
Therefore,  if initially we have $f(x)=-\Delta g(x)$, 
 there is no initial layer in the solution $c^\eps$.
 The next order expansions solve the following system
 \begin{align}
 \partial_t \rho^1_\eps &= \Delta \rho^1_\eps - \nabla \cdot \left( \rho^1_{\eps} \nabla \left(c^0_{\eps,\text{in}}+c^0_{\eps,\text{out}} \right) \right) -  \nabla \cdot \left( \rho^0_{\eps} \nabla c^1_\eps  \right)-  \eps \nabla \cdot \left( \rho^1_{\eps} \nabla c^1_\eps  \right),\\
 \eps \partial_t c^1_\eps &= \Delta  c^1_\eps + \rho^1 - \partial_t c^0_{\eps,\text{out}},
 \end{align}
 with initial conditions
 \[
 \rho^1_\eps(x,0)=0,\quad c^1_\eps(x,0)=0.
 \]
Thus if we can show the boundedness of $\rho^1_\eps$ and $c^1_\eps$,  the validity of the ansatz we proposed will be justified. Further, certain estimates of $c^0_{\eps,\text{in}}$ are needed to show that as $\eps \rightarrow 0$,  the corrections terms vanish and the leading order system converges to the parabolic-elliptic system 
 \begin{align}
\partial_t \rho &= \Delta \rho - \nabla \cdot (\rho  \nabla c), \\
0 & =\Delta c  + \rho,\\
\rho(x,0)& =f(x).
\end{align}

We remark that, the above asymptotic analysis is unclear from a rigorous standpoint, which is beyond the scope of this paper as we focus on designing numerical schemes. Nevertheless, we shall explore numerically the asymptotic behavior of the solutions to give an intuitive justification of the above formal derivation. 

\section{Numerical schemes for the critical case $m=1$}\label{sec:m1}

In this section, we aim to propose numerical schemes for the Keller-Segel system \eqref{eq:rho0} \eqref{eq:c0}, which preserves the parabolic-elliptic limit in the discrete level as $\eps \rightarrow 0$. We show that, under the small data assumption, our scheme (both first and second order) are stable. The spatial discretization is carried out based on a symmetrization of the operators, with which we are able to prove its properties of mass conservation and positivity preservation. The extension to the radially symmetric cases is discussed at the end of this section.

\subsection{A first order semi-discrete scheme and the small data condition}
We first focus on the time discretization and present a semi-discrete scheme for the Keller-Segel equations. Denote $\Delta t$ the time step, then $t^n = n\Delta t $ for $n \in \mathbb{N}$ and $f^n(x)$ represents the numerical approximation to $f(x,t^n)$. Without loss of generality, {we assume homogeneous Dirichlet boundary condition on a bounded Lipschitz domain $\Omega  \subset R^2$} so that no boundary contribution shows up when applying integration by parts. {In this paper, unless specified, all the norms $\| \cdot \|$ denote the $L^2$ norm on the domain $\Omega$.} In theory, other boundary condition can be similarly analyzed and we shall omit them here. 

For stability concern, we want to use implicit method as far as we can, but due to the nonlinearity of the system, this would require a Newton solver that may converge slowly. Here we propose the following semi-discrete scheme: 
\begin{align} \label{eq:rho}
\frac{\rho^{n+1}-\rho^n}{\Delta t} =\Delta \rho^{n+1} - \nabla \cdot (\rho^{n+1} \nabla c^{n+1}) \, ,\\ \label{eq:c}
\eps \frac{c^{n+1}-c^n}{\Delta t} = \Delta c^{n+1}  + \rho^n
\end{align}
to handle the above-mentioned two difficulties. As written, \eqref{eq:c} is just a linear equation for $c^{n+1}$, and thus can be solved cheaply by inverting a symmetric matrix via conjugate gradient or directly using pseudo-spectral method. We will elaborate on it in the next sections. Once $c^{n+1}$ is obtained, (\ref{eq:rho}) reduces to a linear equation for $\rho$ which can also be solved with ease if discretized appropriately. 
Also, we observe that, if we formally take the $\eps \rightarrow 0$ limit with $\Delta t$ fixed, the numerical scheme converges to a semi-discrete method for the limiting parabolic-elliptic model.

To show the stability of this scheme, we have the following theorem. 
\begin{theorem}
{Given a final time $T$, then for $n\Delta t \le T$, assume the numerical solution obtained by  the semi-discrete numerical method \eqref{eq:rho} and \eqref{eq:c} for the Keller-Segel equations satisfies the following technical condition
\begin{equation}\label{est:noblowup}
\Delta t \| \nabla \rho^n \| \le 1, \quad \forall n \ge 0. 
\end{equation}
Then, the method is stable if the small data condition
 \begin{equation}\label{est:data2a}
 \| \rho^{0} \|^2  +\eps \| \nabla c^{0} \|^2   \le 2  e^{-T}.
 \end{equation} 
 is satisfied. }
\end{theorem}

\begin{proof}
Multiply equation \eqref{eq:rho} by $\rho^{n+1} \Delta t$ and integrate with respect to $x$, we get
\begin{multline}\label{eq:rho2}
\frac 1 2 \| \rho^{n+1} \|^2 + \frac 1 2 \| \rho^{n+1}-\rho^n \|^2 - \frac 1 2 \| \rho^n \|^2 + \Delta t \| \nabla \rho^{n+1} \|^2 = -\frac{\Delta t} {2} \left\langle \left(\rho^{n+1}\right)^2, \Delta c^{n+1} \right\rangle\, ,
\end{multline}
where the last term on the left is obtained using integration by parts. Apply the Young's inequatlity, the right hand side of this equation has the following estimate
\[
-\frac {\Delta t} 2 \left\langle \left(\rho^{n+1}\right)^2, \Delta c^{n+1} \right\rangle \le \frac{\Delta t} {4} \| (\rho^{n+1})^2 \|^2 + \frac{\Delta t} {4} \| \Delta c^{n+1} \|^2 .
\]
Next, we multiply equation \eqref{eq:c} by $-\Delta c^{n+1}$ and integrate against $x$. Again with integration by parts, we obtain
\begin{multline} \label{eq:c2}
\frac \eps 2 \| \nabla c^{n+1} \|^2 + \frac \eps 2 \| \nabla c^{n+1}-\nabla c^n \|^2 -  \frac \eps 2 \| \nabla c^n \|^2 + \Delta t \| \Delta c^{n+1}\|^2  
= - \Delta t \left\langle \rho^{n}, \Delta c^{n+1} \right\rangle.
\end{multline}
And the Young's inequality implies
\[
 - \Delta t \left\langle \rho^{n}, \Delta c^{n+1} \right\rangle \le \frac {\Delta t} 2 \| \rho^n \|^2 
+\frac{\Delta t} 2 \| \Delta c^{n+1} \|^2.
\] 
A combination of equation \eqref{eq:rho2} and \eqref{eq:c2} then leads to
\begin{multline} 
\frac 1 2 \| \rho^{n+1} \|^2  +\frac \eps 2 \| \nabla c^{n+1} \|^2 +  \Delta t \| \nabla \rho^{n+1} \|^2 +\frac{\Delta t}{4} \| \Delta c^{n+1}\|^2 
\\
+ \frac 1 2 \| \rho^{n+1}-\rho^n \|^2 + \frac \eps 2 \| \nabla c^{n+1}-\nabla c^n \|^2 
 \\
\le \frac 1 2 (1+\Delta t)\| \rho^n \|^2  + \frac \eps 2 \| \nabla c^n \|^2 +\frac{\Delta t} {4} \| (\rho^{n+1})^2 \|^2. 
\end{multline}
To estimate the nonlinear term $\| (\rho^{n+1})^2 \|^2$ in the two dimensional case, we apply the \emph{Ladyzhenskaya} inequality and get
\[
\| (\rho^{n+1})^2 \|^2 \le 2\| \rho^{n+1}\|^2 \| \nabla \rho^{n+1}\|^2.
 \] 
 Hence we arrive at the following estimate
 \begin{multline}\label{est:add}
\frac 1 2 \| \rho^{n+1} \|^2  +\frac \eps 2 \| \nabla c^{n+1} \|^2 +  \Delta t \left(1-\frac 1 2 \| \rho^{n+1}\|^2 \right) \| \nabla \rho^{n+1} \|^2 
\\
+ \frac {\Delta t}{4} \|\Delta c^{n+1} \|^2 + \frac 1 2 \| \rho^{n+1}-\rho^n \|^2 + \frac \eps 2 \| \nabla c^{n+1}-\nabla c^n \|^2 
 \\
\le \frac 1 2 (1+\Delta t)\| \rho^n \|^2  + \frac \eps 2 \| \nabla c^n \|^2.  
\end{multline}
Thus, if the following condition is satisfied,
\begin{equation} \label{est:data}
1-\frac 1 2 \| \rho^{n+1}\|^2 >0,
\end{equation}
we conclude that
\begin{equation}
 \| \rho^{n+1} \|^2  +\eps \| \nabla c^{n+1} \|^2  \le  (1+\Delta t)\| \rho^n \|^2  + \eps  \| \nabla c^n \|^2.
\end{equation}
The by Gronwall's inequality, if $n \Delta t \le T$, we have
\[
 \| \rho^{n} \|^2  +\eps \| \nabla c^{n} \|^2 \le  e^T \left( \| \rho^{0} \|^2  +\eps \| \nabla c^{0} \|^2  \right). 
 \] 
{We propose that, the presumed condition \eqref{est:data} and the stability estimate require the following small data condition
 \begin{equation}\label{est:data20}
 e^T \left( \| \rho^{0} \|^2  +\eps \| \nabla c^{0} \|^2  \right) \le 2.
 \end{equation}
 Actually, this can be shown by induction. Suppose that, we have shown
  \begin{equation}\label{est:n}
 \frac 1 2 \| \rho^n \|^2  + \frac \eps 2 \| \nabla c^n \|^2 \le e^{n \Delta t -T}
 \end{equation}
 and $(n+1)\Delta t \le T$, then clearly,
 \[
 \frac 1 2 (1+\Delta t)\| \rho^n \|^2  + \frac \eps 2 \| \nabla c^n \|^2 \le e^{(n+1) \Delta t -T} \le 1.
 \]
 If we denote $b^{n+1}= \Delta t \| \nabla \rho^{n+1}\|^2 $, then \eqref{est:add} implies
 \[
 \frac 1 2 \| \rho^{n+1} \|^2 +  b^{n+1} \left(1-\frac 1 2 \| \rho^{n+1}\|^2 \right)  \le 1.
 \]
 Since $b^{n+1}< 1$ due to the technical condition \eqref{est:noblowup}, we conclude that
 \[
 \frac 1 2 \| \rho^{n+1}\|^2  <1,
 \]
 and by \eqref{est:n},  \eqref{est:add} implies 
 \[
 \frac 1 2 \| \rho^{n+1} \|^2  + \frac \eps 2 \| \nabla c^{n+1} \|^2 \le e^{(n+1) \Delta t -T}.
 \]
 This completes the proof.}
 \end{proof}
We end this part with a comment on the asymptotic preserving properties. As $\eps \rightarrow 0$, the scheme {for} the parabolic-parabolic system not only converges to the one for the parabolic-elliptic system, {but also keeps the stability constraint satisfied for fixed $\Delta t$}, as seen from (\ref{est:data20}). This formally justifies that the semi-discrete numerical method \eqref{eq:rho} and \eqref{eq:c} is asymptotically preserving.

\subsection{A conservative and positivity preserving fully discrete scheme}
In this section, we explore in detail the spatial discretizations of  Keller-Segel equations. Note that, naive discretizations of equation \eqref{eq:rho} can easily destroy the positivity of the solution and trigger the instability. Our main idea is to make use of the symmetric formulation of (\ref{eq:sym1}) to design a scheme that guarantees the positivity. 

More specifically, let $M^{n+1} = e^{c^{n+1}}$, and rewrite (\ref{eq:rho}) as
\begin{equation} \label{eqn:1}
\frac{\rho^{n+1}-\rho^n}{\Delta t}=\nabla \cdot \left(M^{n+1} \nabla\left( \frac{\rho^{n+1}}{M^{n+1}} \right) \right),
\end{equation}
where the right hand side is in the form of the Fokker-Planck operator and can be discretized symmetrically \cite{JW11}. In particular, we denote $h^{n+1}=\frac{\rho^{n+1}}{\sqrt{M^{n+1}}}$, and reformulate \eqref{eqn:1} into
\begin{equation}\label{eq:h}
h^{n+1}- \frac{\Delta t}{\sqrt{M^{n+1}}}\nabla \cdot \left( M^{n+1} \nabla \frac{ h^{n+1}}{\sqrt{M^{n+1}}} \right) = \frac{\rho^n}{\sqrt{M^{n+1}}}.
\end{equation}
Such scheme has been shown to preserve positivity. Indeed, since the left hand side is a positive definite operator on $h^{n+1}$, and the right hand side is positive, as long as the spatial discretization preserves the positive definiteness, we can ensure the positivity of $h^{n+1}$.

A fully discrete scheme is in order. Let the computational domain be $[a,b]\times [c,d]$, and we consider uniform spatial mesh with mesh size $\Delta x$ and $\Delta y$. Thus the mesh grid points are $(x_i,y_j)=(a+i \Delta x,c+j \Delta y)$. We apply the following five-point method for spatial decretization { to equation \eqref{eq:h} and \eqref{eq:c}, and get}
\begin{equation}\label{eq:dc}
 \frac{\eps}{\Delta t}  c^{n+1}_{i,j} -D^{n+1}_{i,j} =   \frac{\eps}{\Delta t} c^n_{i,j} +\rho^n_{i,j},
\end{equation}
\begin{equation}\label{eq:dh}
h^{n+1}_{i,j} - \Delta t S^{n+1}_{i,j} = \frac{\rho^n_{i,j}}{\sqrt{M^{n+1}_{i,j}}}.
\end{equation} 
Here, 
\[
D^{n+1}_{i,j} = \frac{1}{\Delta x^2} \left( c^{n+1}_{i-1,j} -2  c^{n+1}_{i,j}+ c^{n+1}_{i+1,j} \right)+\frac{1}{\Delta y^2} \left( c^{n+1}_{i,j-1} -2  c^{n+1}_{i,j}+ c^{n+1}_{i,j+1} \right),
\]
\begin{align*}
S^{n+1}_{ij} = \frac{1}{\Delta x^2 \sqrt{M^{n+1}_{i,j}}}  \sqrt{M^{n+1}_{i+1,j}M^{n+1}_{i,j}} \left( \frac{h^{n+1}_{i+1,j}}{\sqrt{M^{n+1}_{i+1,j}}}-\frac{h^{n+1}_{i,j}}{\sqrt{M^{n+1}_{i,j}}} \right)  \\
-\frac{1}{\Delta x^2 \sqrt{M^{n+1}_{ij}}}   \sqrt{M^{n+1}_{i,j}M^{n+1}_{i-1,j}} \left( \frac{h^{n+1}_{i,j}}{\sqrt{M^{n+1}_{i,j}}}-\frac{h^{n+1}_{i-1,j}}{\sqrt{M^{n+1}_{i-1,j}}} \right)   \\
+\frac{1}{\Delta y^2 \sqrt{M^{n+1}_{i,j}}}  \sqrt{M^{n+1}_{i,j+1}M^{n+1}_{i,j}} \left( \frac{h^{n+1}_{i,j+1}}{\sqrt{M^{n+1}_{i,j+1}}}-\frac{h^{n+1}_{i,j}}{\sqrt{M^{n+1}_{i,j}}} \right) \\
-\frac{1}{\Delta y^2 \sqrt{M^{n+1}_{i,j}}}  \sqrt{M^{n+1}_{i,j}M^{n+1}_{i,j-1}} \left( \frac{h^{n+1}_{i,j}}{\sqrt{M^{n+1}_{i,j}}}-\frac{h^{n+1}_{i,j-1}}{\sqrt{M^{n+1}_{i,j-1}}} \right). 
\end{align*}

When $\Delta x=\Delta y$, we can simplify the above expression to 
\[
D^{n+1}_{i,j} = \frac{1}{\Delta x^2} \left( c^{n+1}_{i-1,j} + c^{n+1}_{i+1,j} + c^{n+1}_{i,j-1}+ c^{n+1}_{i,j+1}  -4  c^{n+1}_{i,j}\right),
\]
\[
S^{n+1}_{ij} = \frac{1}{\Delta x^2} \left( h^{n+1}_{i-1,j} + h^{n+1}_{i+1,j} + h^{n+1}_{i,j-1}+ h^{n+1}_{i,j+1}  -\frac{\sum_{d_1=\pm 1,d_2=\pm 1}\sqrt{M^{n+1}_{i+d_1,j+d_2}}}{\sqrt{M^{n+1}_{i,j}}}  h^{n+1}_{i,j}\right).
\]
In the end, $\rho^{n+1}_{i,j}$ is easily obtained via
\[
\rho^{n+1}_{i,j}=h^{n+1}_{i,j} \sqrt{M^{n+1}_{i,j}}.
 \] 
 Multiply \eqref{eq:dh} by $\sqrt{M^{n+1}_{i,j}}$ and sum over $(i,j)$, we get
 \[
 \sum_{i,j} \rho^{n+1}_{i,j} - \Delta t \sum_{i,j} \sqrt{M^{n+1}_{i,j}} S^{n+1}_{i,j} = \sum_{i,j} \rho^{n}_{i,j}. 
 \]
Notice that
 \begin{multline*}
 \sum_{i,j} \sqrt{M^{n+1}_{i,j}} S^{n+1}_{i,j}  
  =  \sum_{i,j}  \frac{1}{\Delta x^2} \left( \sqrt{M^{n+1}_{i,j} }h^{n+1}_{i+1,j}- \left(\sqrt{M^{n+1}_{i+1,j} }+\sqrt{M^{n+1}_{i-1,j} } \right)h^{n+1}_{i,j}  +\sqrt{M^{n+1}_{i,j} }h^{n+1}_{i-1,j} \right) \\
   +  \sum_{i,j}  \frac{1}{\Delta y^2} \left( \sqrt{M^{n+1}_{i,j} }h^{n+1}_{i,j+1}- \left(\sqrt{M^{n+1}_{i,j+1} }+\sqrt{M^{n+1}_{i,j-1} } \right)h^{n+1}_{i,j}  +\sqrt{M^{n+1}_{i,j} }h^{n+1}_{i,j-1} \right) =0\, ,
 \end{multline*}
 which implies the conservation of mass in the discrete level, i.e., 
  \[
 \sum_{i,j} \rho^{n+1}_{i,j}= \sum_{i,j} \rho^{n}_{i,j}. 
 \]
For positivity, we have the following result. 
\begin{theorem}
{Suppose initially we have $\rho^k_{i,j} \ge 0$ for $k=0$, then the five point scheme \eqref{eq:dc} and \eqref{eq:dh} guarantees 
\[
\rho^n_{i,j} \ge 0 , \quad \text{ for}\,\, \,\,n\ge 1.
\] }
\end{theorem}
The proof is standard and is similar to some existing results, the readers may consult \cite{JW11} for details.

To conclude the discussions on the first order scheme, we would like to give the following remarks,
\begin{enumerate}
\item Given that $c^k_{i,j} \ge 0$ for $k=0,1$ and  appropriate boundary conditions for $c^\eps$,  we can show the positivity of $c^n_{i,j}$ $\forall n\in \N^+$, $\forall i,j$.

\item Other spatial discretization may apply to this semi-discrete system. Especially, the $c^\eps$ equation can easily be solved by pseudo-spectral method. It is worth emphasizing that the positivity of $\rho^n_{i,j}$ is independent of the positivity of  $c^n_{i,j}$. Hence, one has more freedom to solve the $c$ equation.

\item This scheme can be easily extended to multi-species models, as will be shown in Section~\ref{sec:numerics}.
\end{enumerate}


\subsection{A second order scheme} 
The scheme presented above can be directly extended to second order. As the spatial discretization builded upon the center difference is already second order accurate, we just focus on the second order time discretization, which can be accomplished using the backward difference formula (BDF). Specifically, the semi-discrete scheme reads
\begin{align} 
\frac{1}{\Delta t}\left(\frac 3 2\rho^{n+1}-2 \rho^n + \frac 1 2 \rho^{n-1} \right) =\Delta \rho^{n+1} - \nabla \cdot (\rho^{n+1} \nabla c^{n+1}) \label{eq:rho3}
\\ 
\frac{\eps}{\Delta t}\left(\frac 3 2 c^{n+1}-2 c^n + \frac 1 2 c^{n-1} \right) = \Delta c^{n+1}+ 2\rho^n-\rho^{n-1}. \label{eq:c3}
\end{align}
Again, as in the first order scheme, no nonlinear solver is needed: one can solve for $c^{n+1}$ from \eqref{eq:c3} and then $\rho^{n+1}$ from (\ref{eq:rho3}).

A similar stability result is available. 
\begin{theorem}
{Given a final time $T$, then for $n\Delta t \le T$, assume the numerical solution obtained by the second order semi-discrete numerical method \eqref{eq:rho3} and \eqref{eq:c3} for the Keller-Segel equations  satisfies the following technical condition
\begin{equation}\label{est:noblowup2}
\Delta t \| \nabla \rho^n \| \le 1, \quad \forall n \ge 0. 
\end{equation}
Then, the method is stable if the small data condition
 \begin{equation} \label{est:data2b}
\frac 1 4 \| \rho^1 \|^2 +\frac \eps 4 \| \nabla c^1 \|^2+ \frac 1 4   \| 2\rho^{1}-\rho^{0} \|^2+\frac \eps 4  \| 2\nabla c^{1}-\nabla c^{0} \|^2 + \Delta t \|\rho^{0} \|^2   \le \frac{1}{2} {e^{-20T}}
 \end{equation} 
 is satisfied.} 
\end{theorem}

\begin{proof}
Multiply equation \eqref{eq:rho3} by $\rho^{n+1} \Delta t$ and integrate with respect to $x$, by integration by parts, we get
\begin{multline}\label{eq:rho4}
\frac 1 4 \| \rho^{n+1} \|^2- \frac 1 4 \| \rho^n \|^2 + \frac 1 4 \| 2\rho^{n+1}-\rho^n \|^2  -\frac 1 4   \| 2\rho^{n}-\rho^{n-1} \|^2 + \frac 1 4 \| \rho^{n+1}-2 \rho^n +\rho^{n-1}\|^2  \\
 + \Delta t \| \nabla \rho^{n+1} \|^2 = -\frac{\Delta t} {2} \left\langle \left(\rho^{n+1}\right)^2, \Delta c^{n+1} \right\rangle.
\end{multline}
By the Young's inequality, the right hand can be estimated as
\[
-\frac 1 2 \left\langle \left(\rho^{n+1}\right)^2, \Delta c^{n+1} \right\rangle \le \frac{\Delta t} {4} \| (\rho^{n+1})^2 \|^2 + \frac{\Delta t} {4} \| \Delta c^{n+1} \|^2 .
\]
Again, by the \emph{Ladyzhenskaya} inequality, we get
\[
\| (\rho^{n+1})^2 \|^2 \le 2\| \rho^{n+1}\|^2 \| \nabla \rho^{n+1}\|^2.
 \] 
 we multiply equation \eqref{eq:c3} by $-\Delta c^{n+1}$ and integrate with respect to $x$. With integration by parts, we obtain that
\begin{multline} \label{eq:c4}
\frac \eps 4 \| \nabla c^{n+1} \|^2  -  \frac \eps 4 \| \nabla c^n \|^2  + \frac \eps 4  \| 2\nabla c^{n+1}-\nabla c^n \|^2 -  \frac \eps 4  \| 2\nabla c^{n}-\nabla c^{n-1} \|^2 \\ 
+ \frac \eps 4  \| \nabla c^{n+1}-2\nabla c^n+c^{n-1} \|^2 + \Delta t \| \Delta c^{n+1}\|^2 
= - \Delta t \left\langle 2\rho^{n} - \rho^{n-1}, \Delta c^{n+1} \right\rangle.
\end{multline}
And the Young's inequality implies,
\begin{align*}
 - \Delta t \left\langle 2\rho^{n}-\rho^{n-1}, \Delta c^{n+1} \right\rangle & \le \frac {\Delta t} 2 \| 2\rho^n -\rho^{n-1}\|^2 
+\frac{\Delta t} 2 \| \Delta c^{n+1} \|^2 \\
& \le {4\Delta t} \|\rho^n \|^2 + {\Delta t} \|\rho^{n-1} \|^2 +\frac{\Delta t} 2 \| \Delta c^{n+1} \|^2.
\end{align*}
Here we used the fact that $\| a +b\|^2 \le 2 \|a\|^2 +2 \| b\|^2$.
Adding equation \eqref{eq:rho4} and \eqref{eq:c4}, we get,
\begin{multline*}
\frac 1 4 \| \rho^{n+1} \|^2- \frac 1 4 \| \rho^n \|^2 +\frac \eps 4 \| \nabla c^{n+1} \|^2  -  \frac \eps 4 \| \nabla c^n \|^2 + \frac 1 4 \| 2\rho^{n+1}-\rho^n \|^2  -\frac 1 4   \| 2\rho^{n}-\rho^{n-1} \|^2 \\
 + \frac \eps 4  \| 2\nabla c^{n+1}-\nabla c^n \|^2 -  \frac \eps 4  \| 2\nabla c^{n}-\nabla c^{n-1} \|^2  + \frac 1 4 \| \rho^{n+1}-2 \rho^n +\rho^{n-1}\|^2+ \frac \eps 4  \| \nabla c^{n+1}-2\nabla c^n+c^{n-1} \|^2 \\
 \Delta t \left(1-\frac 1 2 \| \rho^{n+1}\|^2 \right) \| \nabla \rho^{n+1} \|^2   + \frac {\Delta t}{4} \|\Delta c^{n+1} \|^2  \le  {4\Delta t} \|\rho^n \|^2 + {\Delta t} \|\rho^{n-1} \|^2.
\end{multline*}
Assume that $\rho^{0}$ and $c^0$ are given by initial conditions, and $\rho^1$ and $c^1$ are computed by a first order numerical scheme. For $N\in \N^+$, $N\ge 2$, with $N\Delta t \le T$, we sum up the above equations for $n=1,\cdots, N-1$, and get
\begin{multline*}
\frac 1 4 \| \rho^{N} \|^2- \frac 1 4 \| \rho^1 \|^2 +\frac \eps 4 \| \nabla c^{N} \|^2  -  \frac \eps 4 \| \nabla c^1 \|^2 + \frac 1 4 \| 2\rho^{N}-\rho^{N-1} \|^2  -\frac 1 4   \| 2\rho^{1}-\rho^{0} \|^2 \\
 + \frac \eps 4  \| 2\nabla c^{N}-\nabla c^{N-1} \|^2 -  \frac \eps 4  \| 2\nabla c^{1}-\nabla c^{0} \|^2   + \sum_{n=1}^{N-1} \frac 1 4 \| \rho^{n+1}-2 \rho^n +\rho^{n-1}\|^2 \\
 + \sum_{n=1}^{N-1}  \frac \eps 4  \| \nabla c^{n+1}-2\nabla c^n+c^{n-1} \|^2   + \sum_{n=1}^{N-1}  \Delta t \left(1-\frac 1 2 \| \rho^{n+1}\|^2 \right) \| \nabla \rho^{n+1} \|^2 \\
  + \sum_{n=1}^{N-1} \frac {\Delta t}{4} \|\Delta c^{n+1} \|^2 \le   {4\Delta t} \|\rho^{N-1} \|^2  + \sum_{n=1}^{N-2}  {5 \Delta t} \|\rho^n \|^2 + {\Delta t}\|\rho^{0} \|^2.
\end{multline*}   
Therefore, if the following condition holds: 
\begin{equation} \label{est:data2}
1-\frac 1 2 \| \rho^{n+1}\|^2 >0, \quad \textrm{ for } n = 1 , \cdots, N-1
\end{equation}
we can conclude that,
\begin{align*}
\frac 1 4 \| \rho^{N} \|^2 +\frac \eps 4 \| \nabla c^{N} \|^2 & \le    {4\Delta t} \|\rho^{N-1} \|^2  +\sum_{n=1}^{N-2}  {5 \Delta t} \|\rho^n \|^2 + C_0 \\
 & \le     \sum_{n=1}^{N-1}  {5 \Delta t} \|\rho^n \|^2 + C_0  \\
 & \le {20 \Delta t}   \sum_{n=1}^{N-1} \left( \frac 1 4 \| \rho^{n} \|^2 +\frac \eps 4 \| \nabla c^{n} \|^2 \right) + C_0.
\end{align*}
where
\[
C_0= \frac 1 4 \| \rho^1 \|^2 +\frac \eps 4 \| \nabla c^1 \|^2+ \frac 1 4   \| 2\rho^{1}-\rho^{0} \|^2+\frac \eps 4  \| 2\nabla c^{1}-\nabla c^{0} \|^2 + \Delta t \|\rho^{0} \|^2.
\]
By induction, we have
\[
\frac 1 4 \| \rho^{N} \|^2 +\frac \eps 4 \| \nabla c^{N} \|^2  \le (1+20\Delta t)^{N-2} ({20 \Delta t} a_1 + C_0),
\]
where
\[
a_1= \frac 1 4 \| \rho^{1} \|^2 +\frac \eps 4 \| \nabla c^{1} \|^2.
\]
Obviously, $a_1 \le C_0$, and thus we have
\[
({20 \Delta t} a_1 + C_0) \le C_0 (1+{20 \Delta t})\, ,
\]
which implies
\[
\frac 1 4 \| \rho^{N} \|^2 +\frac \eps 4 \| \nabla c^{N} \|^2  \le {e^{20T} }C_0.
\]

Subsequently, the following condition is sufficient to guarantee the small data estimate \eqref{est:data2}:
\[
{e^{20T}} C_0 \le \frac 1 2.
\]
{Similar to the first order case, this condition implies the stability estimate, which can be shown by induction.}
\end{proof}

We would remark that, the small data conditions \eqref{est:data2a} \eqref{est:data2b} are not necessary conditions, and are made primarily due to technical issues. In our numerical simulations, we observe that unless the exact solutions to the Keller-Segel equations blow up, the numerical methods do not exhibit unstable behavior.

\subsection{Radially symmetric cases}
This section is devoted to the radially symmetric case. Recall the first order semi-discrete scheme
\begin{align} \label{eq:rho1-2}
\frac{\rho^{n+1}-\rho^n}{\Delta t} =\Delta \rho^{n+1} - \nabla \cdot (\rho^{n+1} \nabla c^{n+1}), \\ \label{eq:c1-2}
\eps \frac{c^{n+1}-c^n}{\Delta t} = \Delta c^{n+1}  + \rho^n\, .
\end{align}
If we confine ourselves to the radially symmetric case, we can write $ \rho (x)= \rho(r)$ and  $c(x)=c(r)$, and simplify the above semi-discrete scheme to 
\begin{align} \label{eq:rhor}
\frac{\rho^{n+1}-\rho^n}{\Delta t} &=\frac{1}{r} \frac{\partial }{\partial r} \left(r \frac{\partial}{\partial r} \rho^{n+1} \right)- \frac{1}{r} \frac{\partial }{\partial r} \left(r \rho^{n+1} \frac{\partial}{\partial r} c^{n+1}\right), \\ \label{eq:cr}
\eps \frac{c^{n+1}-c^n}{\Delta t} &= \frac{1}{r} \frac{\partial }{\partial r} \left(r \frac{\partial}{\partial r} c^{n+1} \right)  + \rho^n, \\
\frac{\partial}{\partial r} \rho^{n+1}(0)&=0, \quad \frac{\partial}{\partial r} c^{n+1} (0)=0.
\end{align}
Then our task is to propose a numerical scheme to this system that is both conservative and positivity preserving. 

If the computation domain is an anulus $a<r<b$, where $0<a<b$, it may be convenient to introduce an auxiliary variable $s=\log r$, or equivalently $r=e^s$, and we have 
\begin{align} \label{eq:rhor2}
e^{2s} \frac{\rho^{n+1}-\rho^n}{\Delta t} &= \frac{\partial^2}{\partial s^2} \rho^{n+1}- \frac{\partial }{\partial s} \left( \rho^{n+1} \frac{\partial}{\partial s} c^{n+1}\right), \\ \label{eq:cr2}
\eps e^{2s} \frac{c^{n+1}-c^n}{\Delta t} &= \frac{\partial^2}{\partial s^2} c^{n+1}+ e^{2s} \rho^n. 
\end{align}
Clearly, we can rewrite \eqref{eq:rhor2} in the following conservative form
\[
e^{2s} \frac{\rho^{n+1}-\rho^n}{\Delta t} = \frac{\partial }{\partial s} \left(e^{c^{n+1}}\frac{\partial}{\partial s} \frac{\rho^{n+1} }{e^{c^{n+1}}}\right).
\]
This system shares the same structure with the one in the cartesian coordinates, and one can design a positivity preserving scheme in the same spirit. However, when $r \rightarrow 0$, $s \rightarrow -\infty$. Therefore, in order the save the information in the vicinity of $r=0$, extra effort is needed when truncating the computational domain in $s$. 

We consider an alternative approach. The key ingredient is the following reformulation of equation  \eqref{eq:rhor} 
\begin{equation}\label{eq:rhor3}
\frac{\rho^{n+1}-\rho^n}{\Delta t} = \frac{1}{r} \frac{\partial }{\partial r} \left(r e^{c^{n+1}}\frac{\partial}{\partial r} \frac{\rho^{n+1} }{e^{c^{n+1}}}\right)\, ,
\end{equation}
Here the computational domain is chosen $r\in[0,L]$, and the mesh size is $\Delta r=\frac{L}{N_r}$, where $N_r \in \mathbb N$ is the number of grid points. $r_j= - \frac{1}{2}\Delta r + j \Delta r$, for {$j=0, 1,\cdots,N_r$}.  {Please note here, $r_0=- \frac{1}{2}\Delta r$ is introduced to handle the following boundary condition at $r=0$.}  We denote the numerical approximation of $f^n(r_j)$ by $f^n_j$. The boundary condition at $r=0$ is
\[
(\rho^n)'(0)=0,\quad (c^n)'(0)=0,
\]
and thus we have
\[
\rho^n_0=\rho^n_1,\quad c^n_0=c^n_1.
\]

For simplicity, we still use $M=e^c$. Then equation \eqref{eq:rhor3} and \eqref{eq:cr} are further discretized into
\begin{align}\label{eq:rhor4}
\frac{\rho^{n+1}_j-\rho^n_j}{\Delta t} &=\frac{1}{\Delta r^2} \frac{1}{r_j}\sqrt{r_j r_{j+1} M^{n+1}_j M^{n+1}_{j+1}} \left( \frac{\rho^{n+1}_{j+1}}{M^{n+1}_{j+1}} - \frac{\rho^{n+1}_j}{M^{n+1}_{j}}\right) \\
\nonumber
&-\frac{1}{\Delta r^2} \frac{1}{r_j}\sqrt{r_j r_{j-1} M^{n+1}_j M^{n+1}_{j-1}} \left( \frac{\rho^{n+1}_{j}}{M^{n+1}_{j}} - \frac{\rho^{n+1}_{j-1}}{M^{n+1}_{j-1}}\right), \\
\label{eq:cr4}
\eps \frac{c_j^{n+1}-c_j^n}{\Delta t} &=\frac{1}{\Delta r^2} \frac{1}{r_j} \sqrt{r_j r_{j+1}} \left( c^{n+1}_{j+1}-c^{n+1}_{j}\right) \\
\nonumber
&- \frac{1}{\Delta r^2} \frac{1}{r_j} \sqrt{r_j r_{j-1}} \left( c^{n+1}_{j}-c^{n+1}_{j-1}\right)+\rho^n_j.
\end{align}
As always, at every time step, we first solve the equation \eqref{eq:cr4} for $c^{n+1}_j$ and then equation \eqref{eq:rhor4} for $\rho^{n+1}_j$. 

Multiply \eqref{eq:rhor4} by $r_j$ and sum over $j$, we can similarly show that
\[
\sum_j r_j \rho^{n+1}_j = \sum_j r_j \rho^{n}_j\, ,
\]
which preserves the discrete mass in the polar coordinates. Moreover, similar to the case in Cartesian coordinates, we can show that the fully discrete scheme \eqref{eq:rhor4} and \eqref{eq:cr4} preserves positivity of $\rho^{n+1}_j$. Indeed, suppose $\rho^n_j \ge 0$, we can recast equation \eqref{eq:rhor4} as
\[
\rho^{n+1}_j= \Delta t R^{n+1}_j +\rho^n_j,
\]
where
\begin{align*}
R^{n+1}_j &=\frac{1}{\Delta r^2} \frac{1}{r_j}\sqrt{r_j r_{j+1} M^{n+1}_j M^{n+1}_{j+1}} \left( \frac{\rho^{n+1}_{j+1}}{M^{n+1}_{j+1}} - \frac{\rho^{n+1}_j}{M^{n+1}_{j}}\right) \\
&-\frac{1}{\Delta r^2} \frac{1}{r_j}\sqrt{r_j r_{j-1} M^{n+1}_j M^{n+1}_{j-1}} \left( \frac{\rho^{n+1}_{j}}{M^{n+1}_{j}} - \frac{\rho^{n+1}_{j-1}}{M^{n+1}_{j-1}}\right).
\end{align*}
If we assume that $\frac{\rho^{n+1}_j}{M^{n+1}_{j}}$ achieves its mininum when $j=j'$ with $\frac{\rho^{n+1}_{j'}}{M^{n+1}_{j'}}<0$, the from the above formulation $R_{j'}^{n+1}>0$ which and thus $\rho_{j'}>0$, leading to a contradiction. Therefore, the positivity is preserved. 


\section{Subcritical case $m> 1$}\label{sec:mn1}
\subsection{Dynamical and steady state}

In this section, we study the 2D Keller-Segel model in the subcritical regime $m>1$
\begin{align}
\partial_t \rho^\eps &= \Delta (\rho^\eps)^m - \nabla \cdot (\rho^\eps \nabla c^\eps), \label{eq:rhom} \\
\eps \partial_t c^\eps & =\Delta c^\eps  + \rho^\eps, \label{eq:cm}\\
\rho^{\eps}(x,0)& =f(x),\quad c^{\eps}(x,0)=g(x).
\end{align}
We first review some properties of this system.

Rewrite equation \eqref{eq:rhom} as
\begin{equation}
\partial_t \rho^\eps =\nabla \cdot \left( \rho^\eps \nabla \mu \right),
\end{equation}
where $\mu$ is the chemical potential
\begin{equation}
\mu = \left \{ 
\begin{split}
&\frac{m}{m-1} (\rho^\eps)^{m-1}-c^\eps,\quad m\ne 1, \\
& \log \rho^\eps - c^\eps, \quad m=1.
\end{split}
  \right. 
\end{equation}
Then the (nonnegative) steady states to this system, which are denoted by $\rho_s^\eps$ and $c_s^\eps$, satisfy the following system in the sense of distribution
\begin{align}
 & \Delta (\rho_s^\eps)^m - \nabla \cdot (\rho_s^\eps \nabla c_s^\eps)=0, \\
& \Delta c_s^\eps  + \rho_s^\eps=0.
\end{align}

To explore the radial symmetry of the steady solution, we define
\begin{equation}
\Omega = \left\{ x \in \R^2; \rho_s^\eps(x)>0 \right\}
\end{equation}
and assume it is connected for simplicity. By \cite{BianLiu2013}, we know  that, when $m\ne 1$, $\rho_s^\eps \in C(\bar \Omega)$ satisfies
\begin{equation}\label{eq:mn1-1}
\begin{split}
 & \frac{m}{m-1} (\rho_s^\eps)^{m-1}-c_s^\eps= \bar c, \quad  x \in \Omega,\\
 & \rho_s^\eps=0, \, x\in \R^2 \setminus \Omega, \quad \rho_s^\eps>0, \, x \in \Omega, \\
& -\Delta c_s^\eps = \rho_s^\eps.
\end{split}
\end{equation}
If we denote $\phi= \frac{m-1}{m}(c^\eps_s+ \bar c)$, then  \eqref{eq:mn1-1} implies
\begin{equation}\label{eq:mn1-2}
\begin{split}
&-\Delta \phi = \frac{m-1}{m} \phi^k, \quad x\in \Omega, \quad k=\frac{1}{m-1},\\
&\phi=0,\, x\in \partial \Omega, \quad \phi>0, \, x \in \Omega. 
\end{split}
\end{equation}
The nonnegative radial classical solution of \eqref{eq:mn1-2} can be written in the form $\phi(x)=\phi(r)$, thus, $\forall a>0$, if we define $L=\{ r ; \phi(r) \ge 0 \}$, $\phi(r)\in C^2([0,L])$ satisfies the following initial value problem
\begin{equation}
\begin{split}
& \phi_{rr}+ \frac 2 r \phi_r = - \frac{m-1}{m} \phi^k, \quad r>0,\quad k=\frac{1}{m-1}, \\
& \phi'(0)=0,\quad \phi(0)=a>0.
\end{split}
\end{equation}
Here, $\phi(r)^k$ is meaningful before it reaches zero.

 When $m= 1$, the steady solution $\rho_s^\eps \in C(\bar \Omega)$ satisfies
\begin{equation}\label{eq:m1-1}
\begin{split}
 & \log \rho_s^\eps-c_s^\eps= \bar c, \quad  x \in \Omega,\\
 & \rho_s^\eps=0, \, x\in \R^2 \setminus \Omega, \quad \rho_s^\eps>0, \, x \in \Omega, \\
& -\Delta c_s^\eps = \rho_s^\eps.
\end{split}
\end{equation}
If we denote $\phi= \log \rho_s^\eps$, then,  \eqref{eq:m1-1} implies
\begin{equation}\label{eq:m1-2}
\begin{split}
&-\Delta \phi =e^\phi, \quad x\in \R^2.
\end{split}
\end{equation}
The nonnegative radial classical solution of \eqref{eq:m1-2} can be written in the form $\phi(x)=\phi(r)$, thus, $\forall a>0$, if we define $L=\{ r ; \phi(r) \ge 0 \}$, $\phi(r)\in C^2([0,L])$ satisfies the following initial value problem
\begin{equation}
\begin{split}
& \phi_{rr}+ \frac 2 r \phi_r = - e^\phi, \quad r>0, \\
& \phi'(0)=0,\quad \phi(0)=a>0.
\end{split}
\end{equation}

\subsection{Numerical scheme}
Similar to the critical case, we first propose the following semi-discrete method for the (2D) Keller-Segel model with exponent $m$,
\begin{align} \label{eq:rhomn}
\frac{\rho^{n+1}-\rho^n}{\Delta t} =\Delta (\rho^{n+1})^m - \nabla \cdot (\rho^{n+1} \nabla c^{n+1}) \\ \label{eq:cmn}
\eps \frac{c^{n+1}-c^n}{\Delta t} = \Delta c^{n+1}+ \rho^n. 
\end{align}

Here a Newton's solver is inevitable due to the nonlinearity on the right hand side. And because of this, the stability analysis can be very complicated. We skip the analysis on this scheme here and instead show substantial numerical evidence to verify the properties of this method to the model especially in the subcritical  cases. 

Another issue of this scheme concerns the positivity. We observe numerically that when $m>1$, this scheme is not necessarily positivity preserving, especially when the solution is compacted supported, or when the diffusion exponent $m$ is large.

 To propose a positivity scheme, recall that equation \eqref{eq:rhom} can be reformulated as
\begin{align}
\partial_t \rho^\eps &= \nabla \cdot \left[ \rho^\eps \exp\left(c^\eps-\frac{m}{m-1}(\rho^\eps)^{m-1}\right) \nabla \exp\left(-c^\eps+\frac{m}{m-1}(\rho^\eps)^{m-1} \right) \right]. \label{eq:rhom2}
\end{align}

Let $M=\exp\left(c^\eps-\frac{m}{m-1}(\rho^\eps)^{m-1}\right)$, then we have equivalently,
\begin{align}
\partial_t \rho^\eps &= \nabla \cdot \left[ \rho^\eps M \nabla \frac{1}{ M}\right]= \nabla \cdot \left[ \rho^\eps M \nabla \frac{\rho^\eps}{\rho^\eps M}\right]. \label{eq:rhom2}
\end{align}
Therefore, we propose the following semi-discrete, semi-implicit scheme 
\begin{align} \label{eq:rhomn2}
\frac{\rho^{n+1}-\rho^n}{\Delta t} = \nabla \cdot \left[ \rho^n M^n \nabla \frac{\rho^{n+1}}{\rho^n M^n}\right], \\ \label{eq:cmn2}
\eps \frac{c^{n+1}-c^n}{\Delta t} = \Delta c^{n+1}  + \rho^n. 
\end{align}
In the radial symmetric case, we write $ \rho^\eps (x)= \rho^\eps(r)$ and  $c^\eps (x)=c^\eps (r)$, and the system \eqref{eq:rhom} and \eqref{eq:cm} rewrite as
\begin{align} \label{eq:rhorm}
\partial_t \rho^\eps &=\frac{1}{r} \frac{\partial }{\partial r} \left(r \frac{\partial}{\partial r} (\rho^{\eps})^m \right)- \frac{1}{r} \frac{\partial }{\partial r} \left(r \rho^{\eps} \frac{\partial}{\partial r} c^{\eps}\right), \\ \label{eq:crm}
\eps \partial_t c^\eps &= \frac{1}{r} \frac{\partial }{\partial r} \left(r \frac{\partial}{\partial r} c^{\eps} \right)+ \rho^\eps, \\
\frac{\partial}{\partial r} \rho^{\eps}(0,t)&=0, \quad \frac{\partial}{\partial r} c^{\eps} (0,t)=0.
\end{align}

Again, we denote by $M=\exp\left(c^\eps-\frac{m}{m-1}(\rho^\eps)^{m-1}\right)$, equation \eqref{eq:rhorm} reformulate to
\begin{align} \label{eq:rhorm2}
\partial_t \rho^\eps &=\frac{1}{r} \frac{\partial }{\partial r} \left(r \rho^\eps M\frac{\partial}{\partial r} \frac{1}{M} \right)=\frac{1}{r} \frac{\partial }{\partial r} \left(r \rho^\eps M\frac{\partial}{\partial r} \frac{\rho^\eps}{\rho^\eps M} \right).
\end{align}
And the corresponding the following semi-discrete, semi-implicit scheme reads
\begin{align} \label{eq:rhormn2}
\frac{\rho^{n+1}-\rho^n}{\Delta t} =\frac{1}{r} \frac{\partial }{\partial r} \left(r \rho^{n} M^n \frac{\partial}{\partial r} \frac{\rho^{n+1}}{\rho^{n} M^n} \right), \\ \label{eq:crmn2}
\eps \frac{c^{n+1}-c^n}{\Delta t} = \Delta c^{n+1}  + \rho^n. 
\end{align}
Similar to this previous cases, we can show that five-point scheme for the semi-discrete system \eqref{eq:rhomn2} \eqref{eq:cmn2} and the centered difference approximation for the semi-discrete system \eqref{eq:rhormn2} \eqref{eq:crmn2} are both conservative and positivity preserving. As the proofs are similar to that of the previous cases, we shall omit them here.

\section{Numerical examples} \label{sec:numerics}
In this section, we present several numerical examples in dimension two. Here periodic boundary condition is used among all examples. The first three examples concern $m=1$ whereas the last one focuses on $m>1$. 
\subsection{Convergence}
First we check the accuracy of the first and second order schemes in cartesian coordinates. Here the initial data takes the following form
\begin{eqnarray}
\rho(x,0) = 4e^{-(x^2+y^2)}, \quad c(x,0) = e^{-(x^2+y^2)/2}, \qquad x\in[-5,5] \quad y \in [-5,5],
\end{eqnarray}
and output time is $t_\textrm{max}=5$. The meshes are chosen $\Delta x = 1, 0.5, 0.25, 0.125$, respectively, and $\Delta t = \Delta x$. The relative error is computed as
\begin{eqnarray} \label{eqn: error1}
error_{\Delta x} = \frac{|| \rho_{\Delta x}(x,t_\textrm{max}) - \rho_{2\Delta x} (x,t_\textrm{max})||_{\ell^1}}{ || \rho_{\Delta x} ||_{\ell^1}},
\end{eqnarray}
and collected in Fig. \ref{fig:conv}. Here a uniform convergence for both first and second order schemes are observed for a wide range of $\eps$.

\begin{figure}[!h]
\includegraphics[width = 0.5\textwidth]{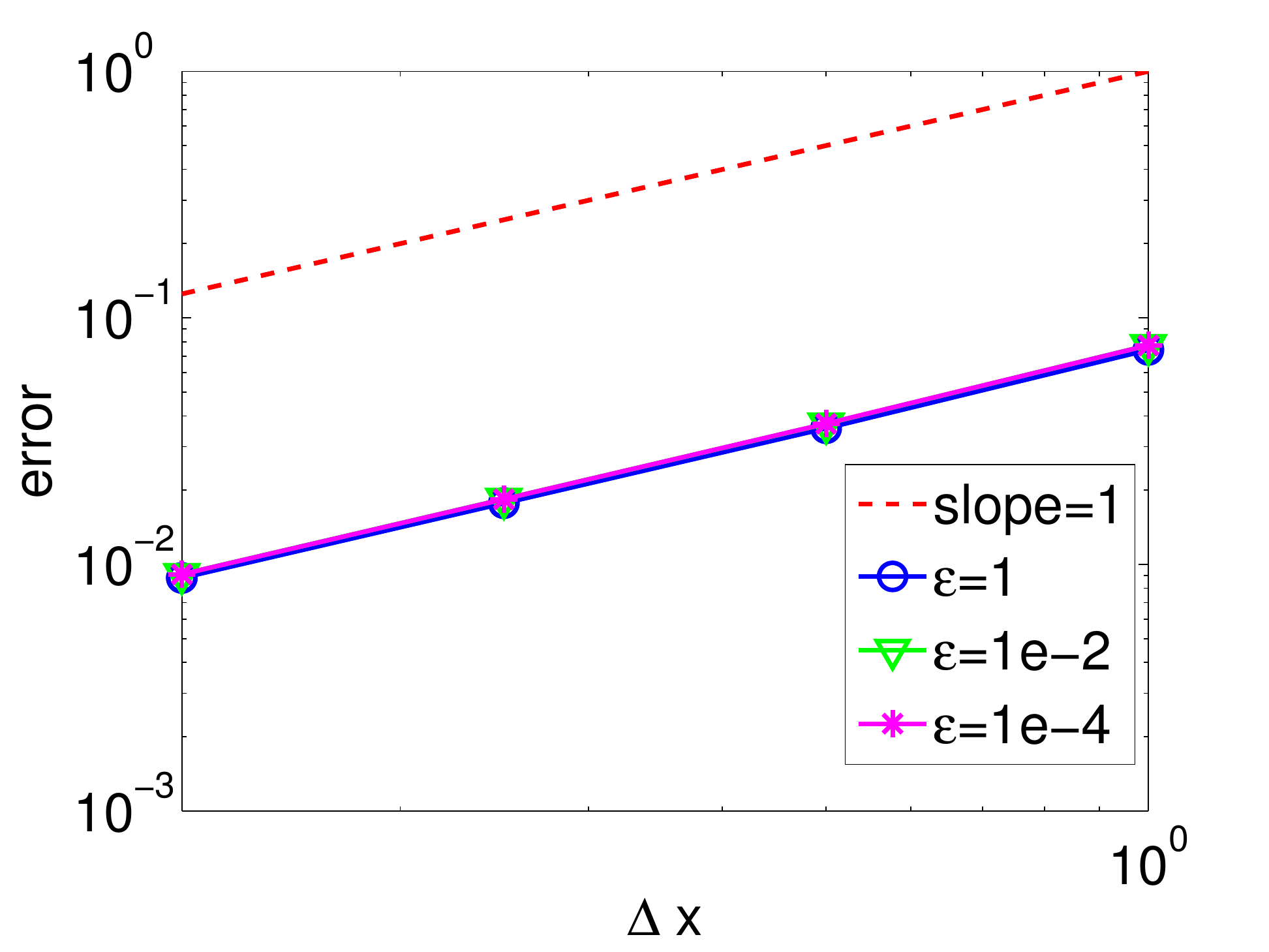}
\includegraphics[width = 0.5\textwidth]{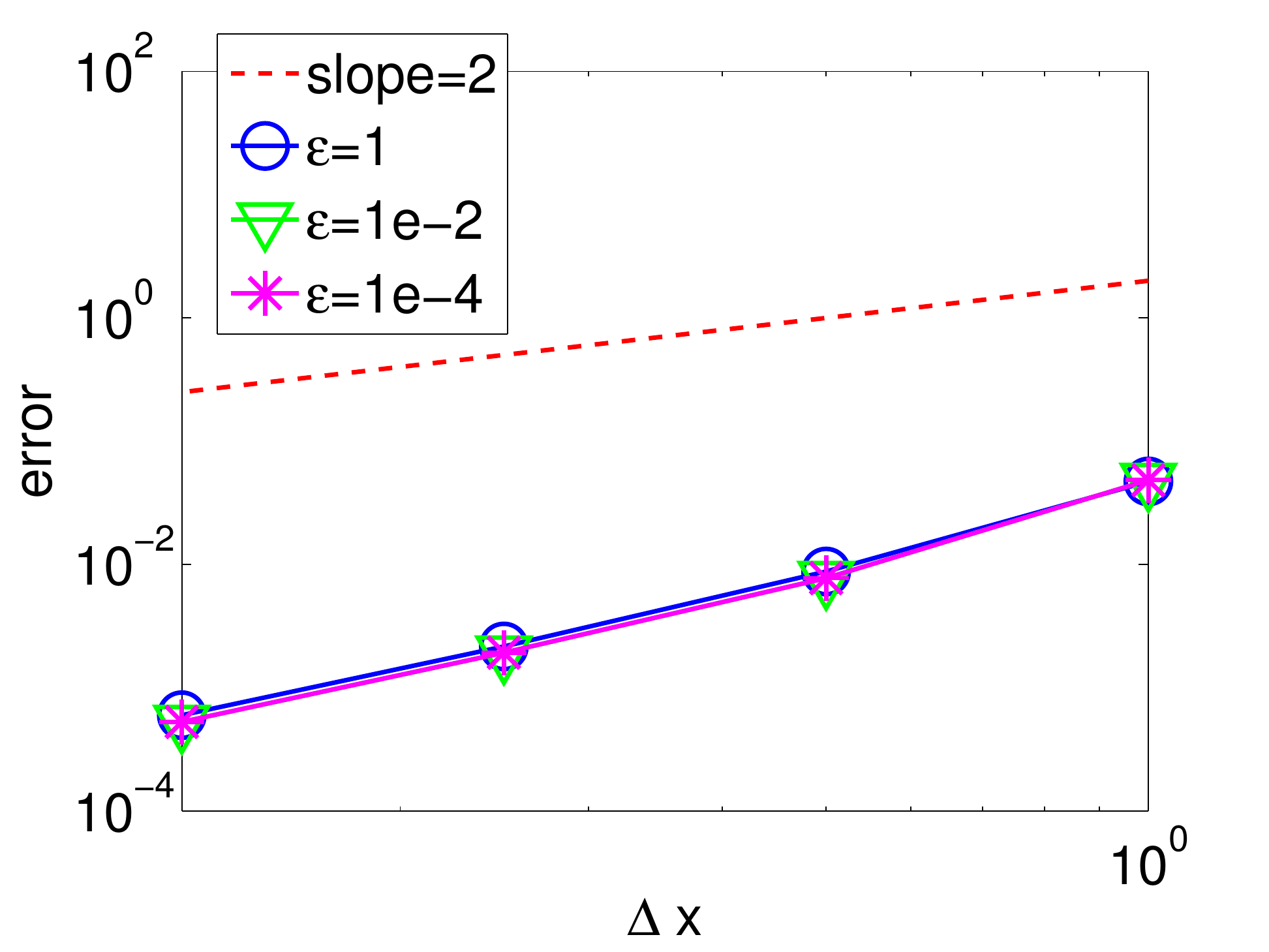}
\caption{Uniform convergence of our schemes: error (\ref{eqn: error1}) versus mesh size $\Delta x = \Delta y$ for different $\eps = 10^{-4}$, $10^{-2}$ and $1$. The red dashed line is a reference with a fixed slope. Left: first order scheme (\ref{eq:dc}) (\ref{eq:dh}). Right: second order scheme (\ref{eq:rho3}) (\ref{eq:c3}).}
\label{fig:conv}
\end{figure}

\subsection{Asymptotic behavior}
Next, we demonstrate the asymptotic behavior of both $\rho$ and $c$. Both asymptotic in small $\eps$ limit and long time limit will be considered. 

\subsubsection{Quasi-static asymptotic behavior} \label{section421}
Denote $\rho^\eps$ and $c^\eps$ the solution to (\ref{eq:dc}) and (\ref{eq:dh}), and $\rho^0$ and $c^0$ the solutions with $\eps =0$, and we compute the $\ell^1$ error in time:
\begin{align}
&\|\rho^\eps(x,y,t) - \rho^0(x,y,t) \|_{\ell^1} = \sum_{i, j} |(\rho^\eps)_{i,j}^n- (\rho^0)_{i,j}^n| \Delta x \Delta y ,\label{eqn:error-asy-rho}
\\
&\|c^\eps(x,y,t) - c^0(x,y,t) \|_{\ell^1} = \sum_{i, j} |(c^\eps)_{i,j}^n- (c^0)_{i,j}^n| \Delta x \Delta y .\label{eqn:error-asy-c}
\end{align}
The initial data is chosen to be 
\begin{align} \label{IC:421}
\rho(x,0) = 400e^{-100(x^2 + y^2)}, \qquad c(x,0) = e^{-50(x^2+y^2)}
\end{align}
such that $\rho(x,0) \neq (1-\Delta)^{-1} f(x,0)$. The results are gathered in Fig.\ref{fig:asymp} for different choices of $\eps$. Here the computational domain is $(x,y) \in [-1,1] \times [-1,1]$ the meshes are $\Delta x = \Delta y  =0.05$, and we use both big time step $\Delta t = 0.05$ and small time step $\Delta t = 5e-4$. It is shown that in $c$, the error undergoes a drastic change at the beginning until it reaches a state after which the errors decrease at the order of $\eps$. This initial period time is independent of our choice of time step, which implies that it is a period of initial layer. After such layer, the error decreases as $\eps$ decreases, and they change at the same order, as suggested in Section 2. On the contrary, the error in $\rho$ varies at the same order of $\eps$ starting from the beginning, which implies the non-existence of initial layer. This transition can be observed even with a coarse time step, as shown in Fig.~\ref{fig:asymp}. To get a closer look at the dynamics in this layer regime, we have a zoom-in plot in the lower left corner are computed using small $\Delta t< 10^{-3}$, less than the smallest $\eps$ we choose here. Then a similar transition discussed above is observed, further confirm the asymptotic behavior of the solutions. 
\begin{figure}[!h]
\includegraphics[width = 0.5\textwidth]{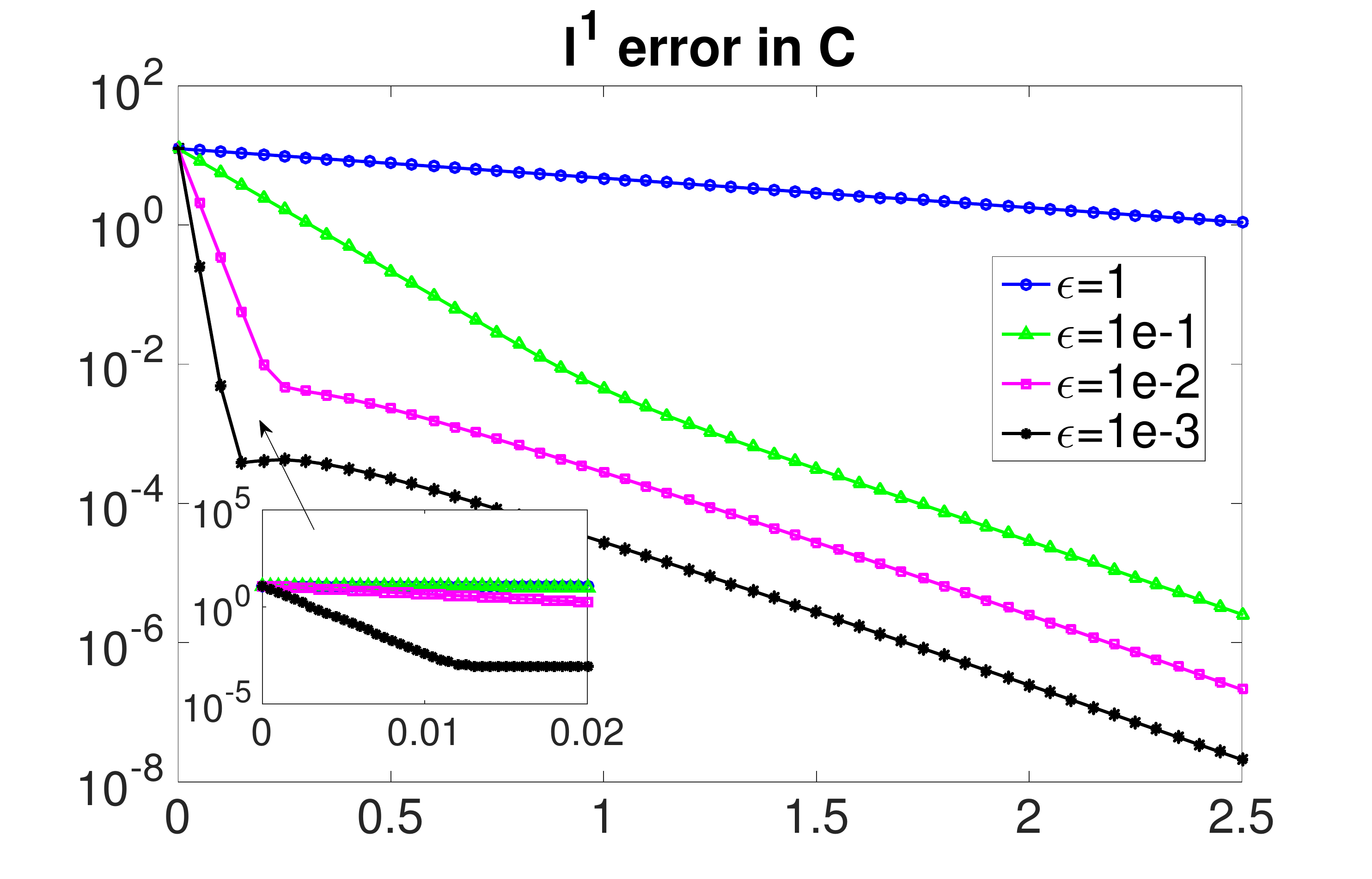}
\includegraphics[width = 0.5\textwidth]{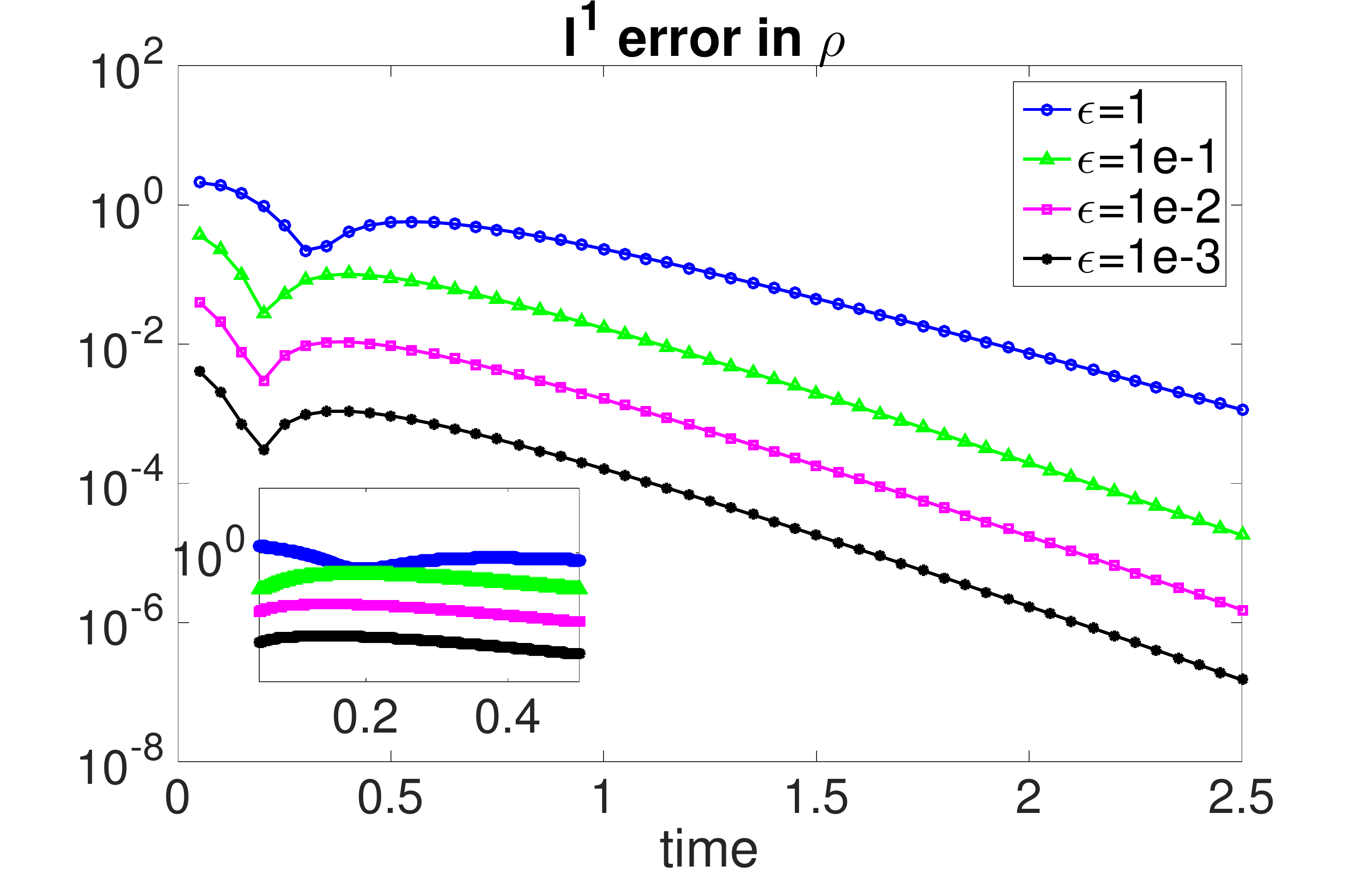}
\caption{Left: $\ell^1$ error in $c$ (\ref{eqn:error-asy-c}). Right:  $\ell^1$ error in $\rho$ (\ref{eqn:error-asy-rho}). Here $\Delta x = \Delta y  =0.05$, $\Delta t = 0.05$ for the big picture and $\Delta t = 5e-4$ in the pictures on the lower left corner.}
\label{fig:asymp}
\end{figure}

\subsubsection{Long time behavior}
Here we briefly compute the free energy at each time. The initial condition is taken the same as in (\ref{IC:421}), and the computational domain, mesh size and time step are kept all the same as in section \ref{section421}. When $\eps=1$, the free energy is defined in (\ref{eqn:energy1}), and (\ref{eq:111}) when $\eps = 0$. In Fig \ref{fig:free-energy}, we plot both cases and observe the decay of energy in time, a property highlighted in \cite{CCH15}. 
\begin{figure}[!h]
\includegraphics[width = 0.5\textwidth]{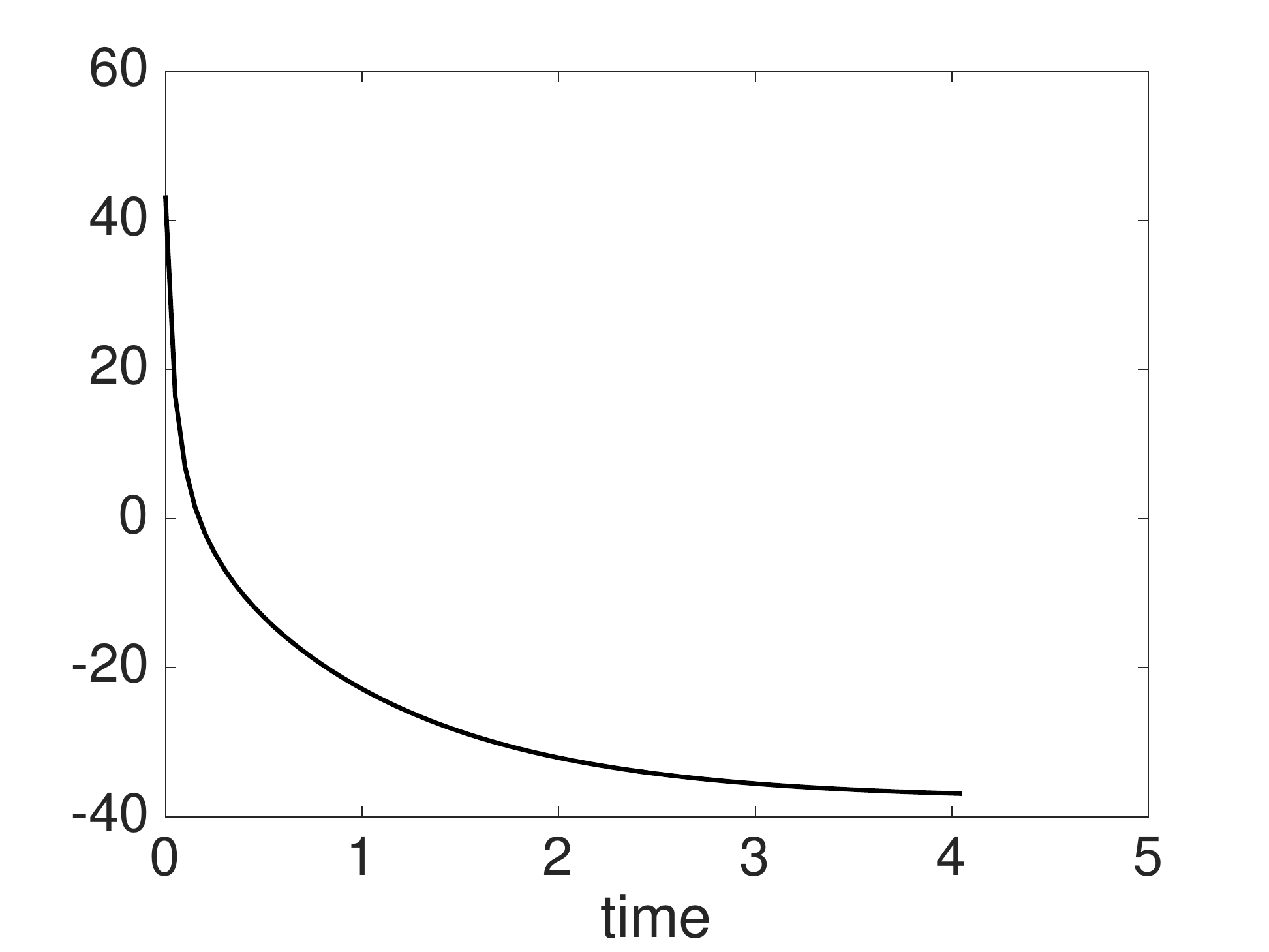}
\includegraphics[width = 0.5\textwidth]{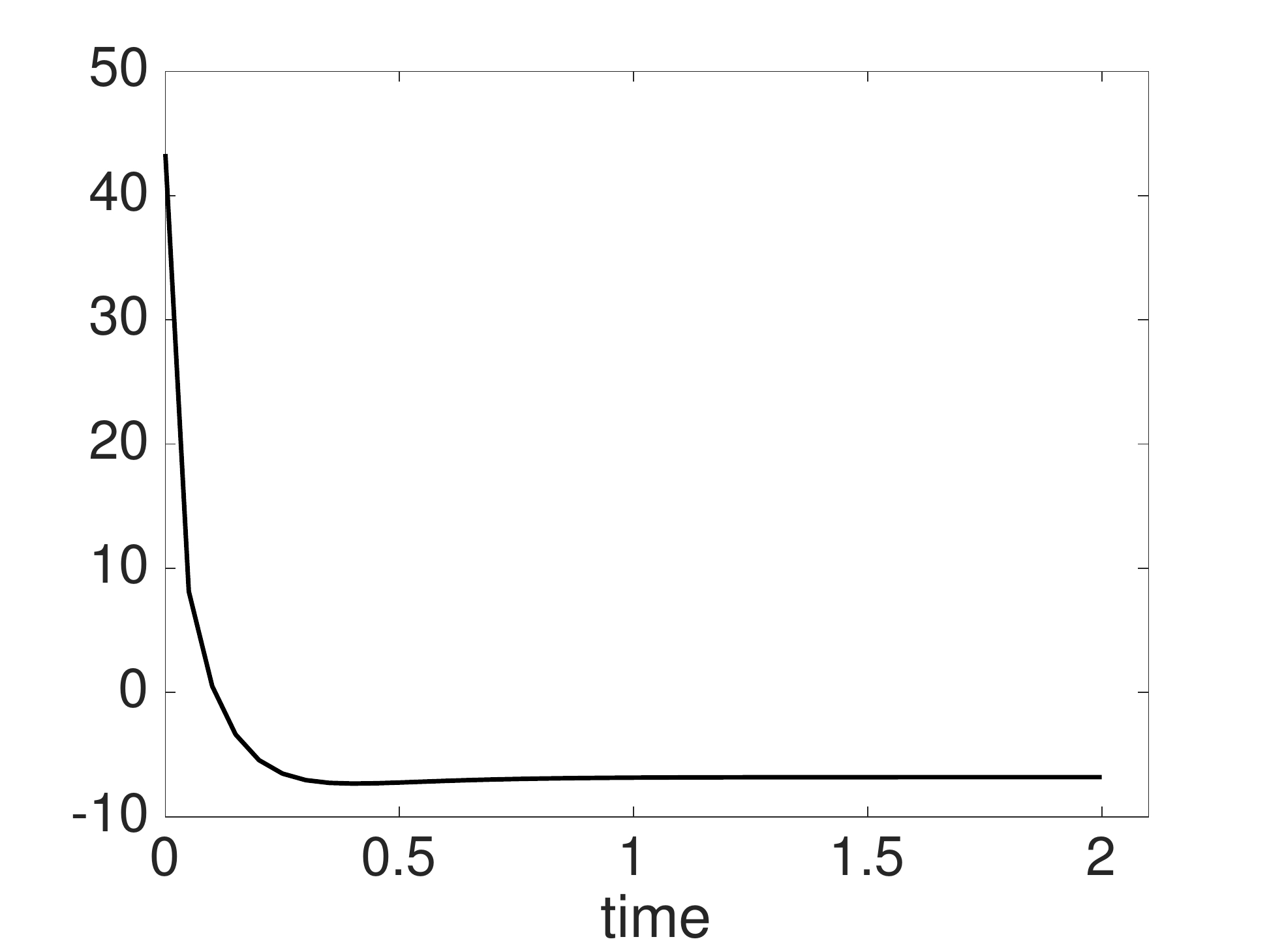}
\caption{Plot of free energy versus time. Left: $\eps=1$ with free energy defined in (\ref{eqn:energy1}). Right: $\eps = 0$ with energy (\ref{eq:111}). Here $\Delta x = \Delta y = \Delta t = 0.05$.}
\label{fig:free-energy}
\end{figure}

\subsection{Blow up}
In this subsection, we focus on the cases when $\rho$ blows up, and show that our schemes, both in cartesian and polar coordinates, are positivity preserving regardless of the choice of $\Delta t$. For the radial symmetric case, consider the following initial data for $\rho(r,0)$
\begin{equation}
\rho(r,0) = 600e^{-60r^2}, \quad r\in[0,2]
\end{equation}
and we choose $c(r,0)$ such that it solves
\begin{equation}
\frac{1}{r} \frac{\partial}{\partial r}\left( r \frac{\partial}{\partial r} c(r,0)\right) - c(r,0) + \rho(r,0) = 0. 
\end{equation}
When $\eps =0$, we plot the profile of $\rho$ at different times in Fig. \ref{fig:blowup_raidal} on the left, and on the right, we show the maximum of $\rho$ with time. Different mesh sizes are used, for the upper figures $\Delta r = 0.025$ and lower figures $\Delta r = 0.00625 $, and $\Delta t = \Delta r/5$. It is interesting to point out that, the maximum amplitude of $\rho$ increases by a factor of $16$ as $\Delta r$ decreases by $1/4$, indicating a blow up of $\rho$ in $\mathcal{O}\left(\frac{1}{\Delta r^2}\right)$ fashion. 
\begin{figure}[!h]
\includegraphics[width = 0.45\textwidth]{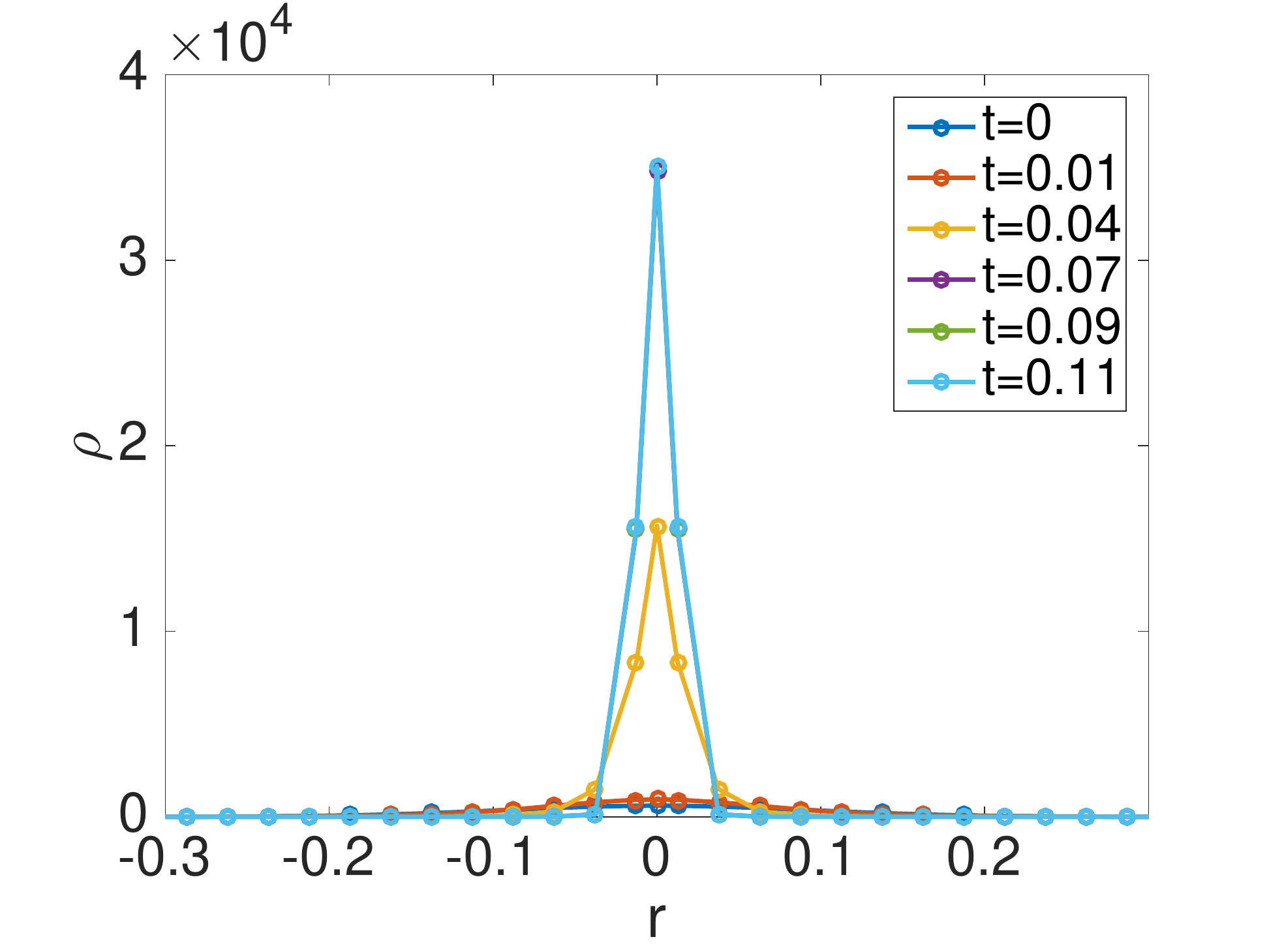}
\includegraphics[width = 0.45\textwidth]{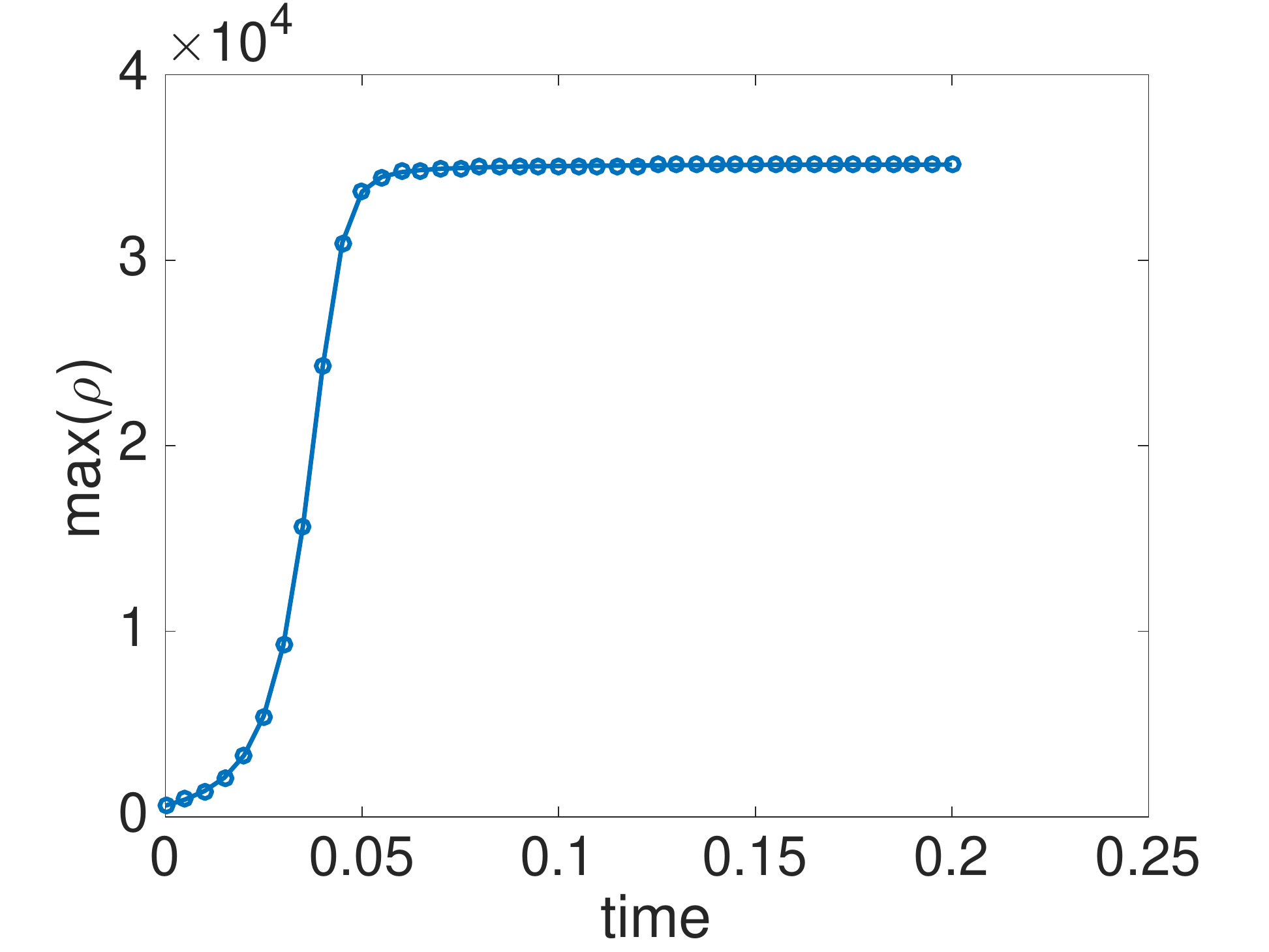}
\\
\includegraphics[width = 0.45\textwidth]{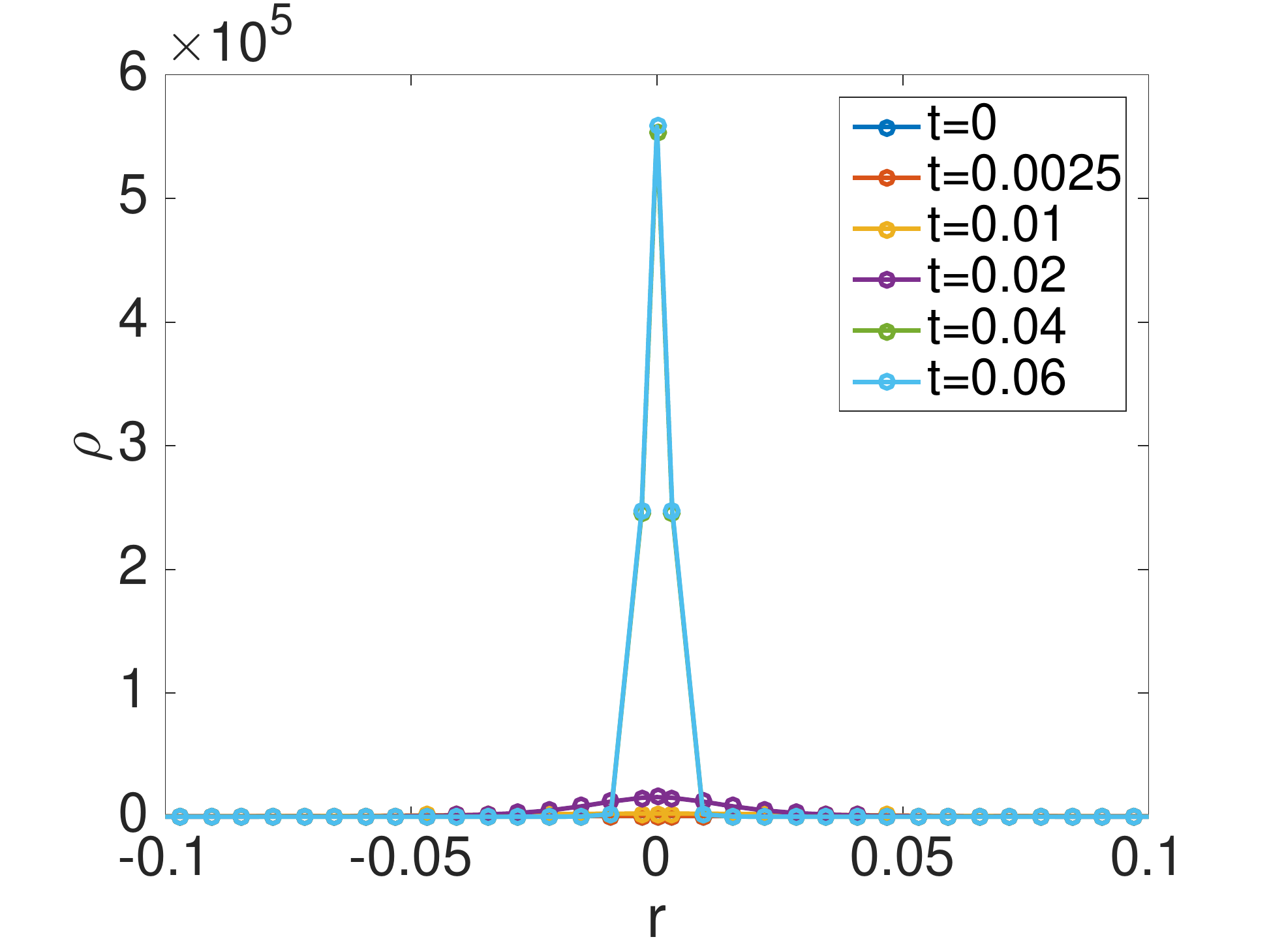}
\includegraphics[width = 0.45\textwidth]{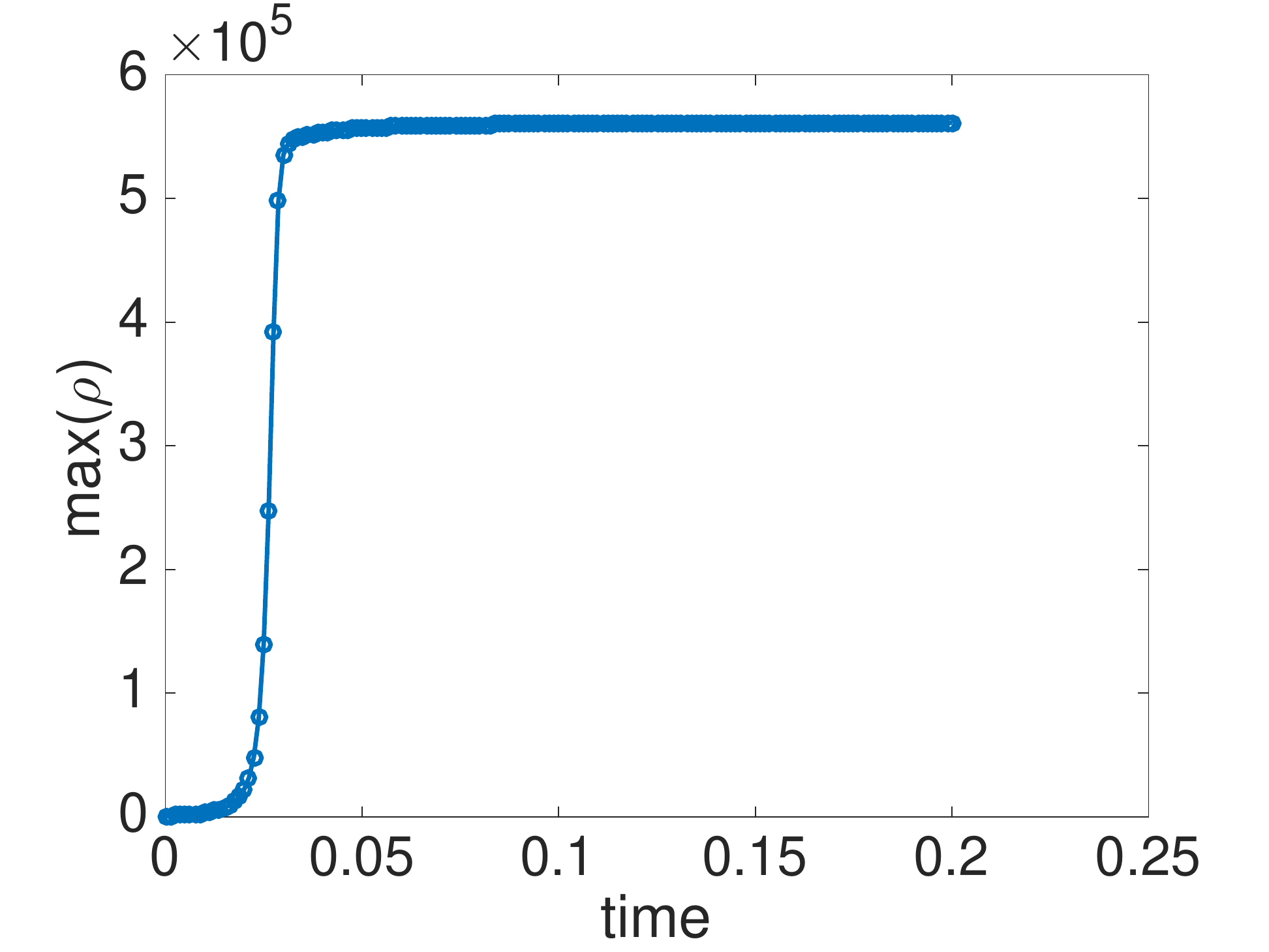}
\caption{Computation of the model in radial symmetric case with $\eps = 0$. Left: the plot of $\rho$ at different times. Right: $\max(\rho)$ versus time. Top: $\Delta r = 0.025$. Bottom: $\Delta r = 0.00625$. $\Delta t = \Delta r/5$.}
\label{fig:blowup_raidal}
\end{figure}

Similarly, in cartesian coordinates, we consider the following initial data
\begin{equation}
\rho(x,y,0) = 600e^{-60(x^2 + y^2)}, \quad (x,y)\in[-4,4]\times[-4,4], \quad c(x,y,0) = 300e^{-30(x^2 + y^2)}.
\end{equation}
In Fig. \ref{fig:blowup_cart} on the left, we plot a slice of solution at $y=0$ for different times, where a trend to blow up is observed. On the right, we plot the maximum magnitude of $\rho$, which is very similar to the one obtained in the radial symmetric case. Also, we observe that this magnitude increases at the order of $\mathcal{O}\left( \frac{1}{\Delta x^2}\right)$. Similar type of blow up is observed in \cite{LiShuYang}.

\begin{figure}[!h]
\includegraphics[width = 0.55\textwidth]{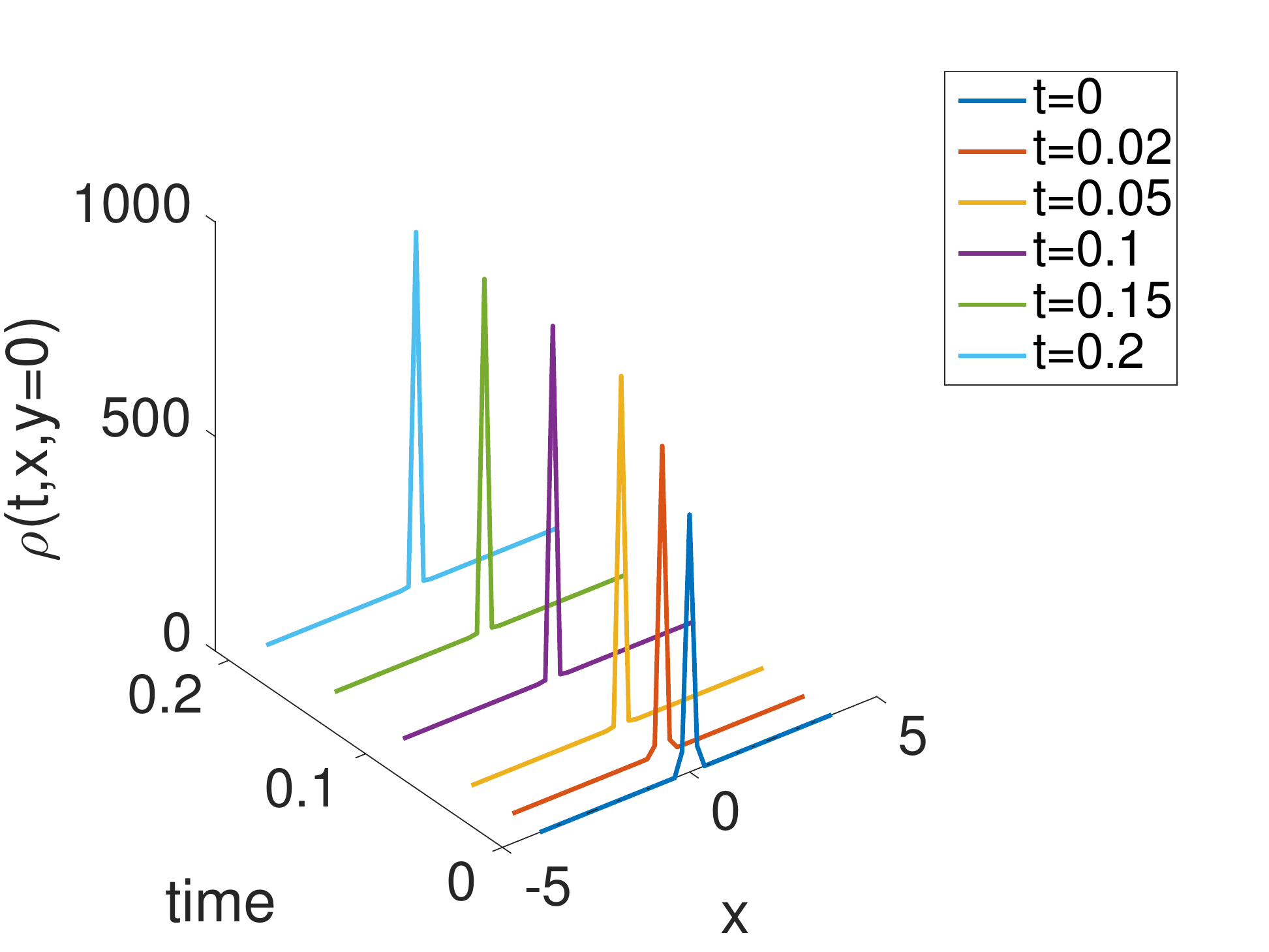}
\includegraphics[width = 0.45\textwidth]{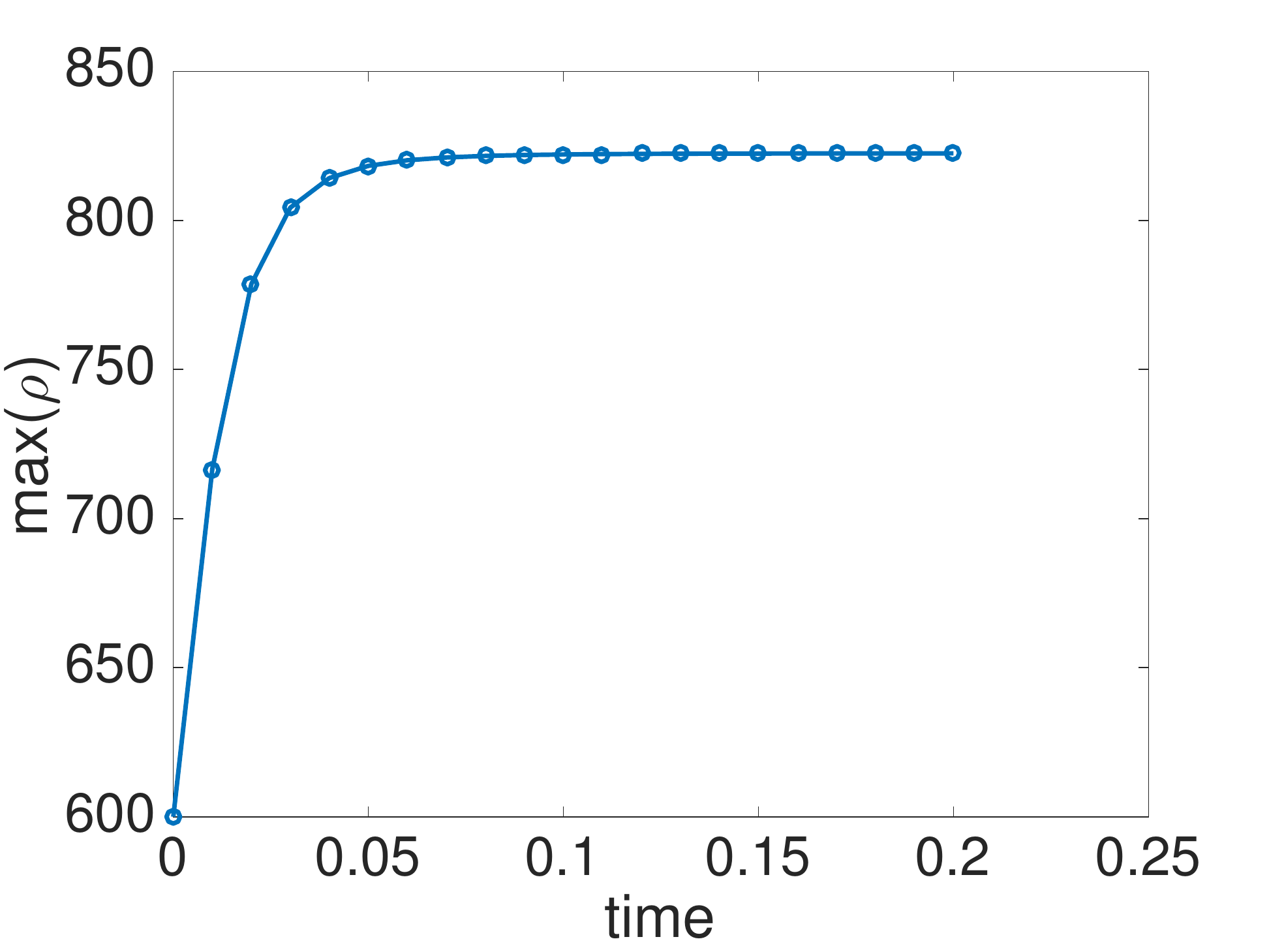}
\\
\includegraphics[width = 0.55\textwidth]{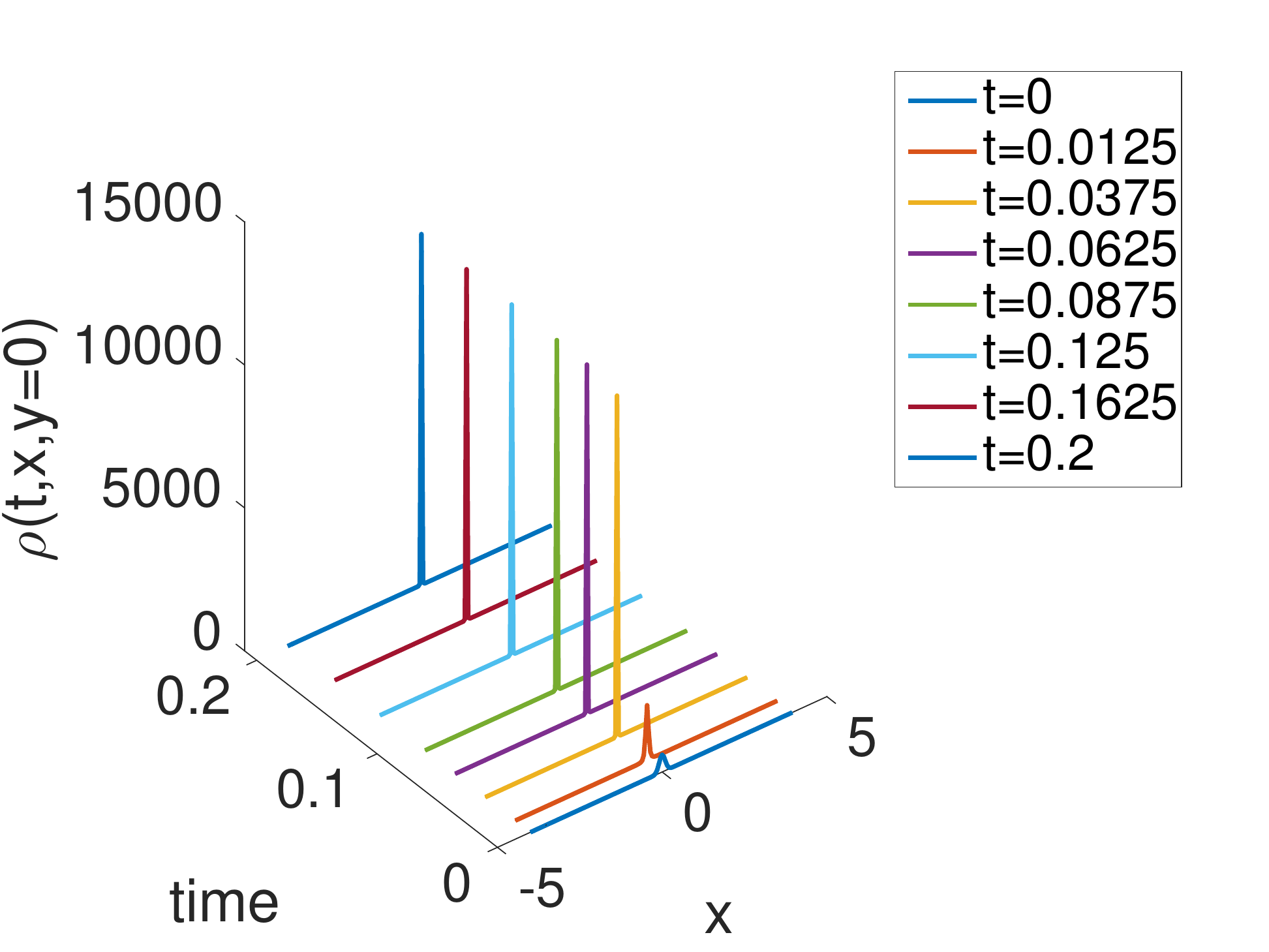}
\includegraphics[width = 0.45\textwidth]{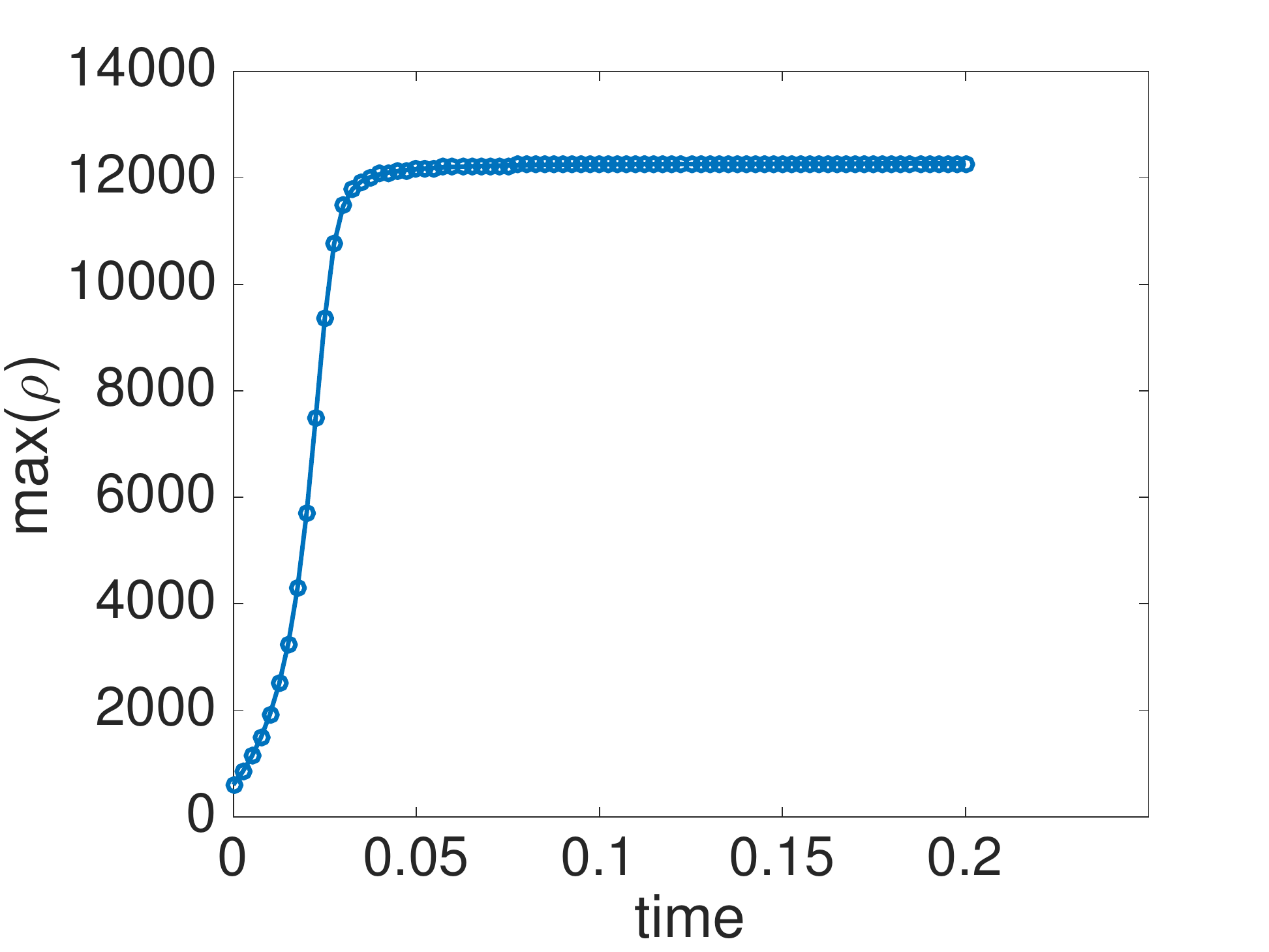}
\caption{Computation of the model in cartesian coordinates. $\eps = 0$. Left: the plot of  a slice of $\rho$ at $y=0$ at different times. Right: $\max(\rho)$ versus time. Top: $\Delta x = \Delta y = 0.2$. Bottom: $\Delta x = \Delta y = 0.05$. $\Delta t = \Delta x/20$.}
\label{fig:blowup_cart}
\end{figure}

\subsection{Subcritical case $m>1$}
This section is devoted to the subcritical case: $m>1$. Our focus will be the limit behavior when $m \rightarrow \infty$. First we consider the `square' initial data in polar coordinates
\begin{equation}
\rho(r,0) = \left\{ \begin{array}{cc}   \rho_0  & r^2\leq 0.1 \\ 0 & \text{elsewhere} \end{array}
\right. \quad c(r,0) = \frac{1}{2}\rho(r,0)
\end{equation} 
displayed in black curve in Fig. \ref{fig:porous_square}, where $\rho_0$ is a constant. The output time is $50$, long enough to produce a solution in steady state. On the left $\rho_0 = 1$, and one sees that as $m$ increases, the steady state solution transits from a smooth, fat bump to a tall sharp square that happens to be the same as the initial profile. This indicates that the steady state, as $m \rightarrow \infty$, tends to converge to the characteristic function with the length of the region determined by the total mass. We then choose $\rho_0 = 0.5$, and similar trends is observed on the right of Fig.~\ref{fig:porous_square}, which confirms the recent result that the steady state in the infinity limit of $m$ tends to the characteristic function; see \cite{CraigKimYao}. 
\begin{figure}[!h]
\includegraphics[width = 0.5\textwidth]{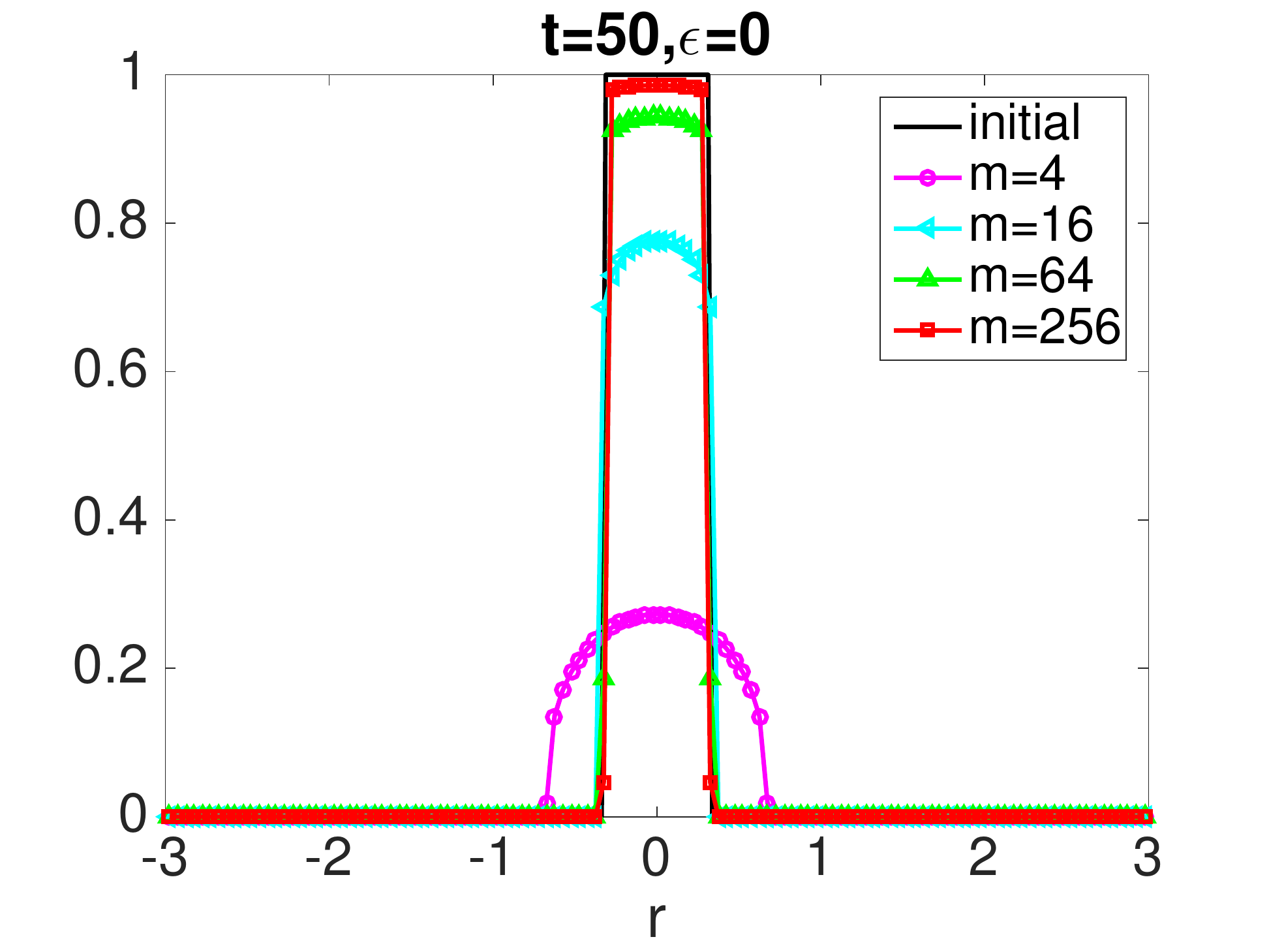}
\includegraphics[width = 0.5\textwidth]{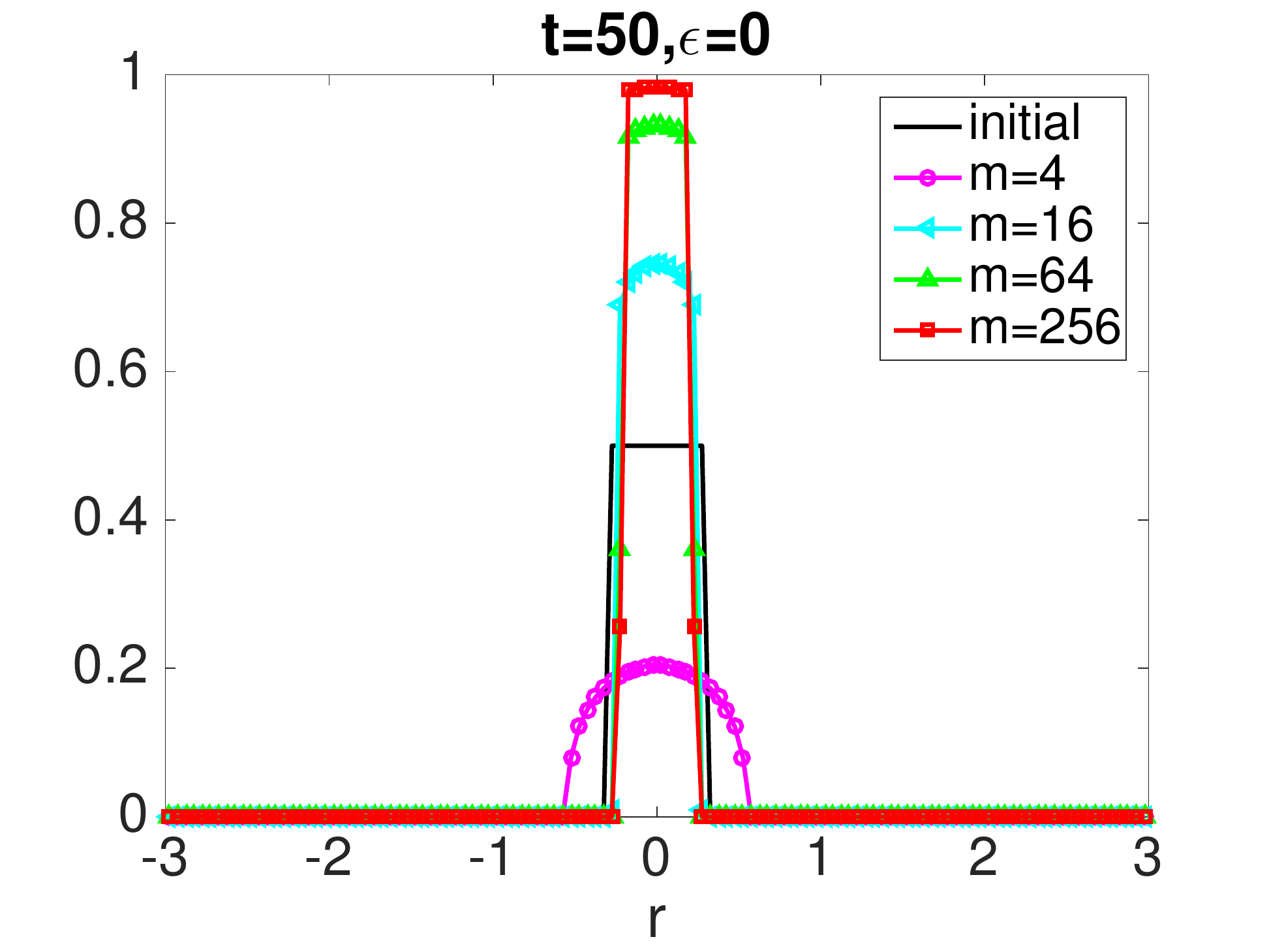}
\caption{Computation of the radial symmetric case (\ref{eq:rhorm}) (\ref{eq:crm}). $\eps = 0$, output time is $t=50$, and the plot of $\rho$ for different $m = 4$, $16$, $64$, and $256$. The black solid curves are the initial profile of $\rho$. Left: $\rho_0 = 1$. Right: $\rho_0 = 0.5$. Here $\Delta r = 0.05$, $\Delta t = 1.25e-4$.}
\label{fig:porous_square}
\end{figure}

To further check the shape of the steady state, we compute the problem in the cartesian grid. First we choose the initial data to be a double annulus, which is radially symmetric, as shown in the upper left of Fig. ~\ref{fig:porous_annulus}:
\begin{equation} \label{ic1}
\rho(x,y,0) = \left\{ \begin{array}{cc}  1  & 0.5< x^2 + y^2 <1 \textrm{ or } 1.5< x^2 + y^2 <2, 
\\ 0 & \textrm{ elsewhere, }
\end{array} \right.
\quad c(x,y,0) = \frac{1}{2} \rho(x,y,0).
\end{equation}
The next two figures display the profile of $\rho$ at later times, both in the top view and in $3D$ view. From these three figures, one sees that the shape of $\rho$, starting out with a double annulus, tends towards a thicker single annulus closer to the origin, and then towards a circle around the origin, which is just a $2D$ analog of the radial symmetric case in the previous test. The last picture in Fig.~\ref{fig:porous_annulus} plots one cross-section of $\rho$ at $x=0$, and the dynamics is the same as we expected. 

\begin{figure}[!h]
\includegraphics[width = 0.5\textwidth]{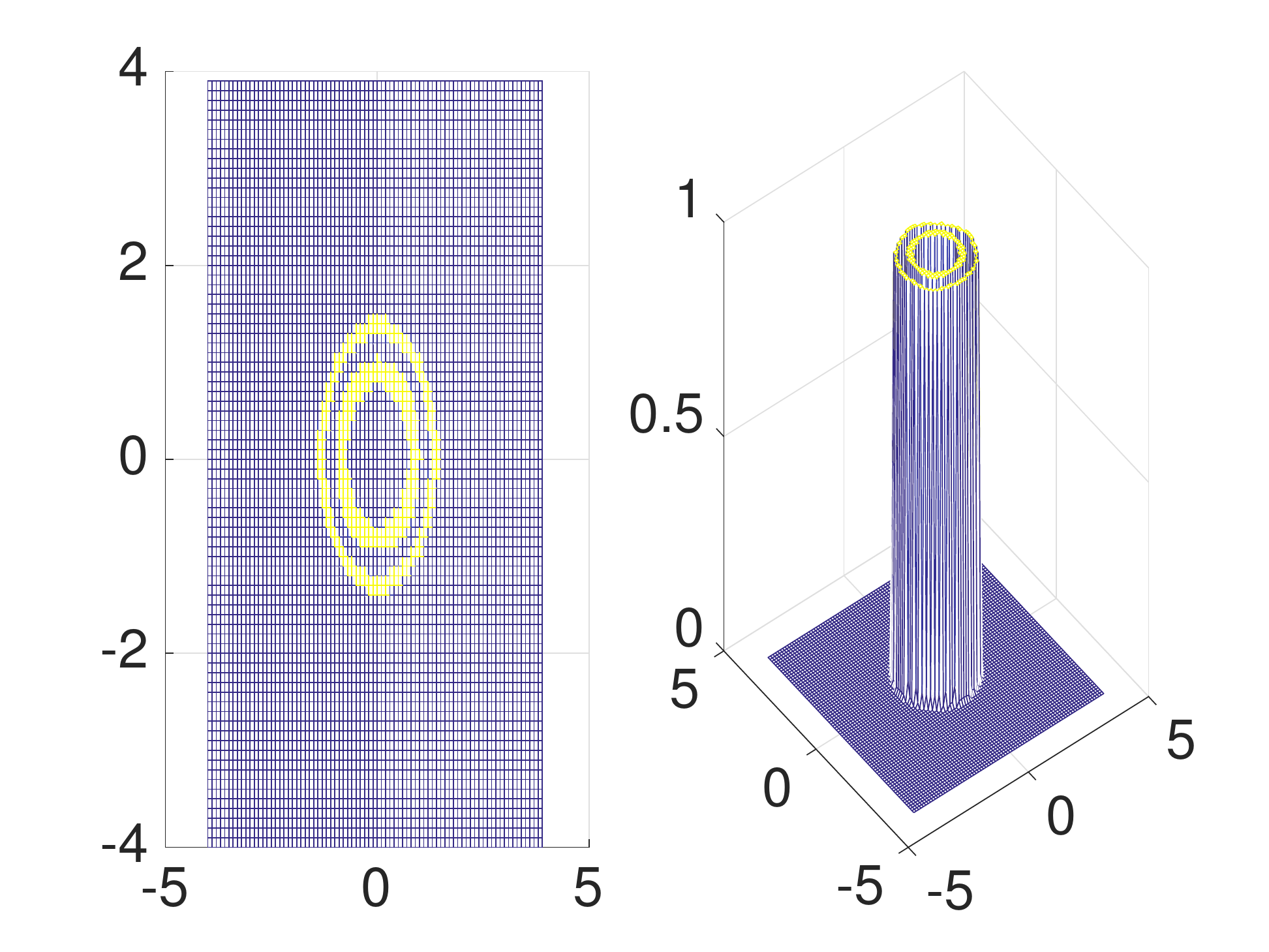}
\includegraphics[width = 0.5\textwidth]{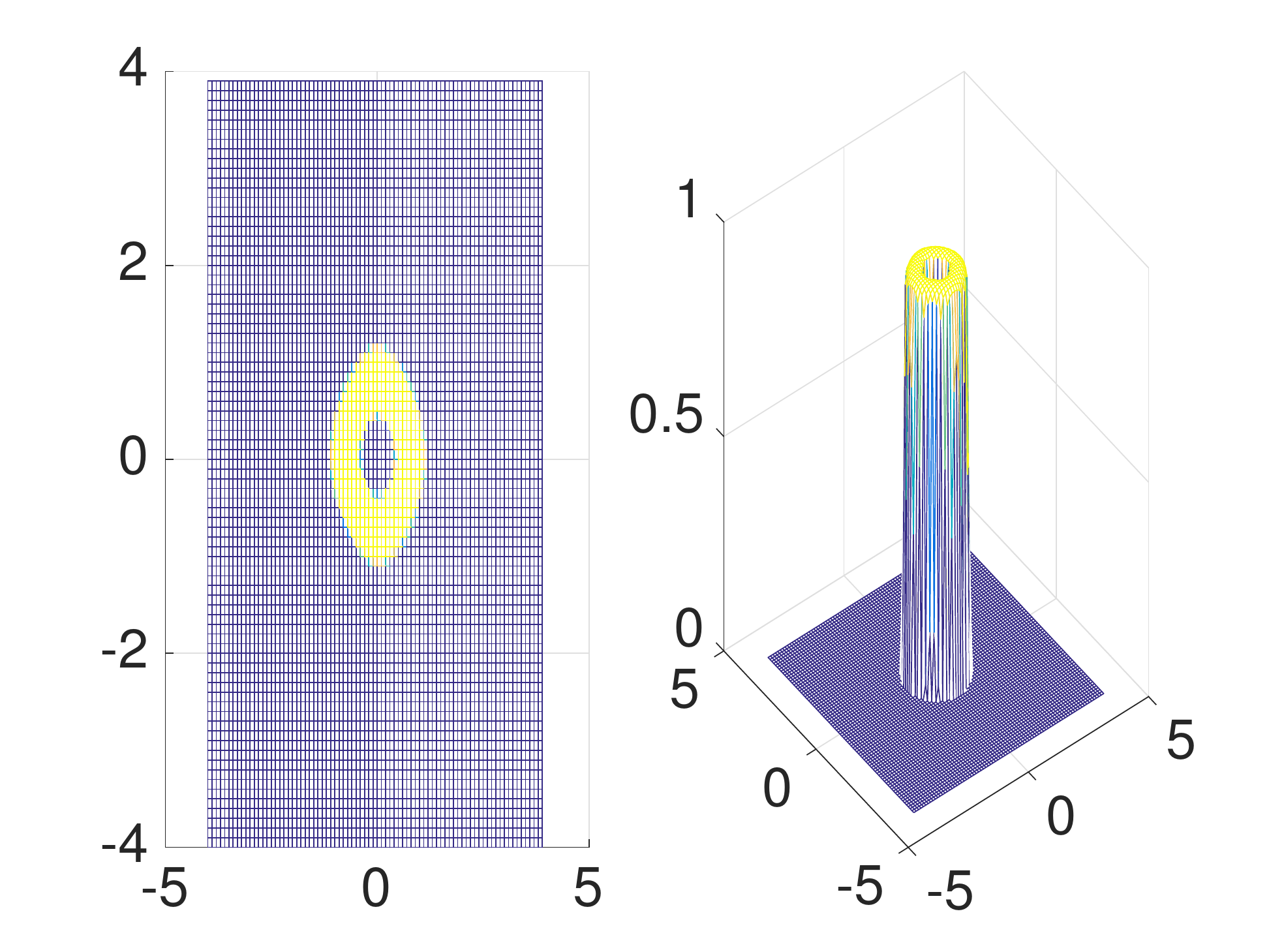}
\\ 
\includegraphics[width = 0.5\textwidth]{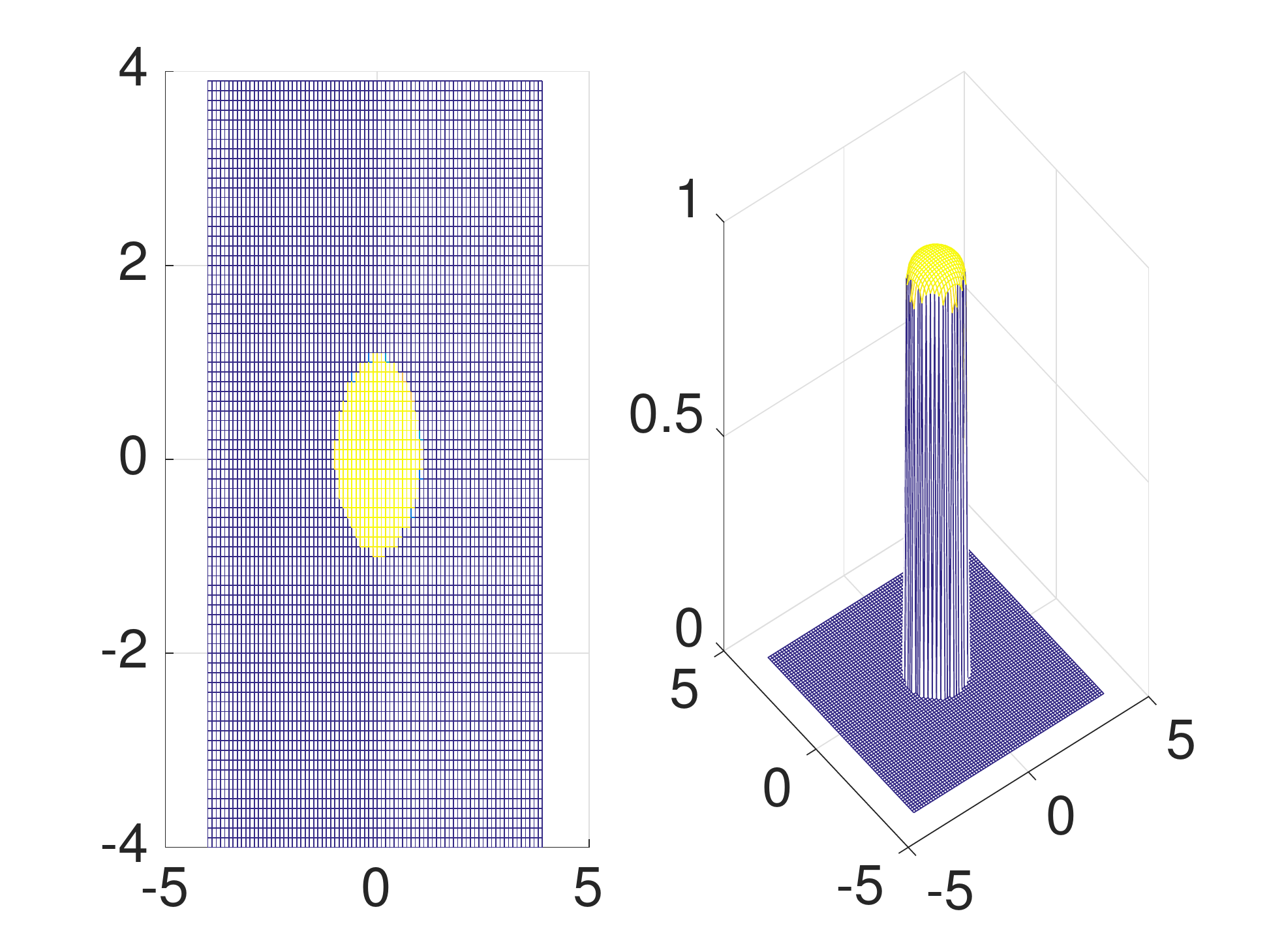}
\includegraphics[width = 0.5\textwidth]{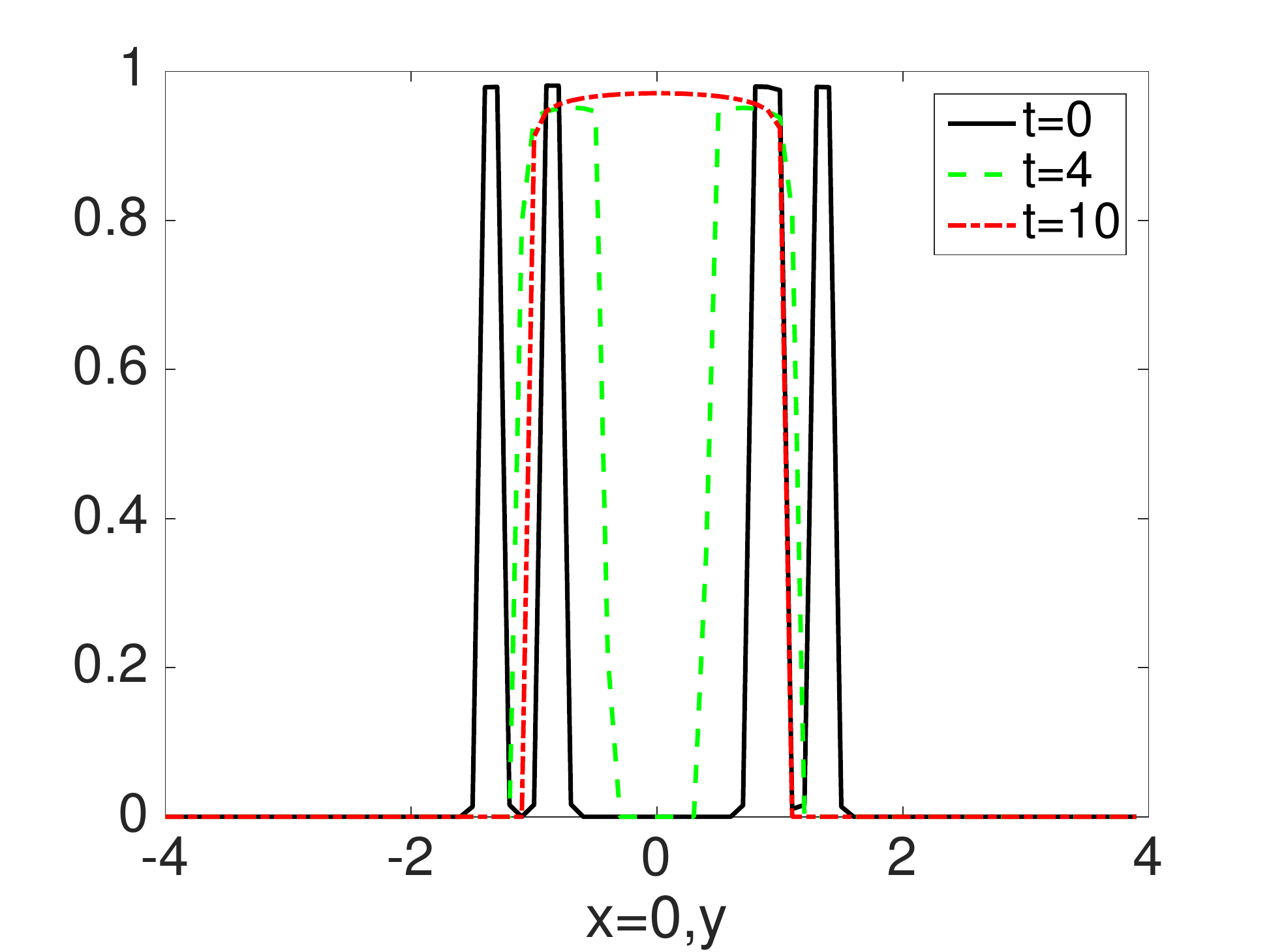}
\caption{Time evolution of model (\ref{eq:rhomn}) (\ref{eq:cmn}) with initial data (\ref{ic1}). $\eps = 0$, $m = 64$. Out put times are: $t=0$ (upper left), $t=4$ (upper right), $t = 10$ (lower left). Lower right: plot of one cross section of $\rho$ at $x=0$.  }
\label{fig:porous_annulus}
\end{figure}

In the end, we consider a case with non-radially symmetric initial data
\begin{align} \label{ic2}
\rho(x,y,0) &= \left\{ \begin{array}{cc}  1  & -1\leq x\leq -0.1, 0.1\leq y \leq 1  \textrm{ or } 0\leq x \leq1, -1\leq y \leq 0, 
\\ 0 & \textrm{ elsewhere, }
\end{array} \right.
\\ c(x,y,0) &= \frac{1}{2} \rho(x,y,0). \label{ic3}
\end{align}
The dynamics is displayed in Fig.~\ref{fig:porous_twobump}. 
\begin{figure}[!h]
\includegraphics[width = 0.5\textwidth]{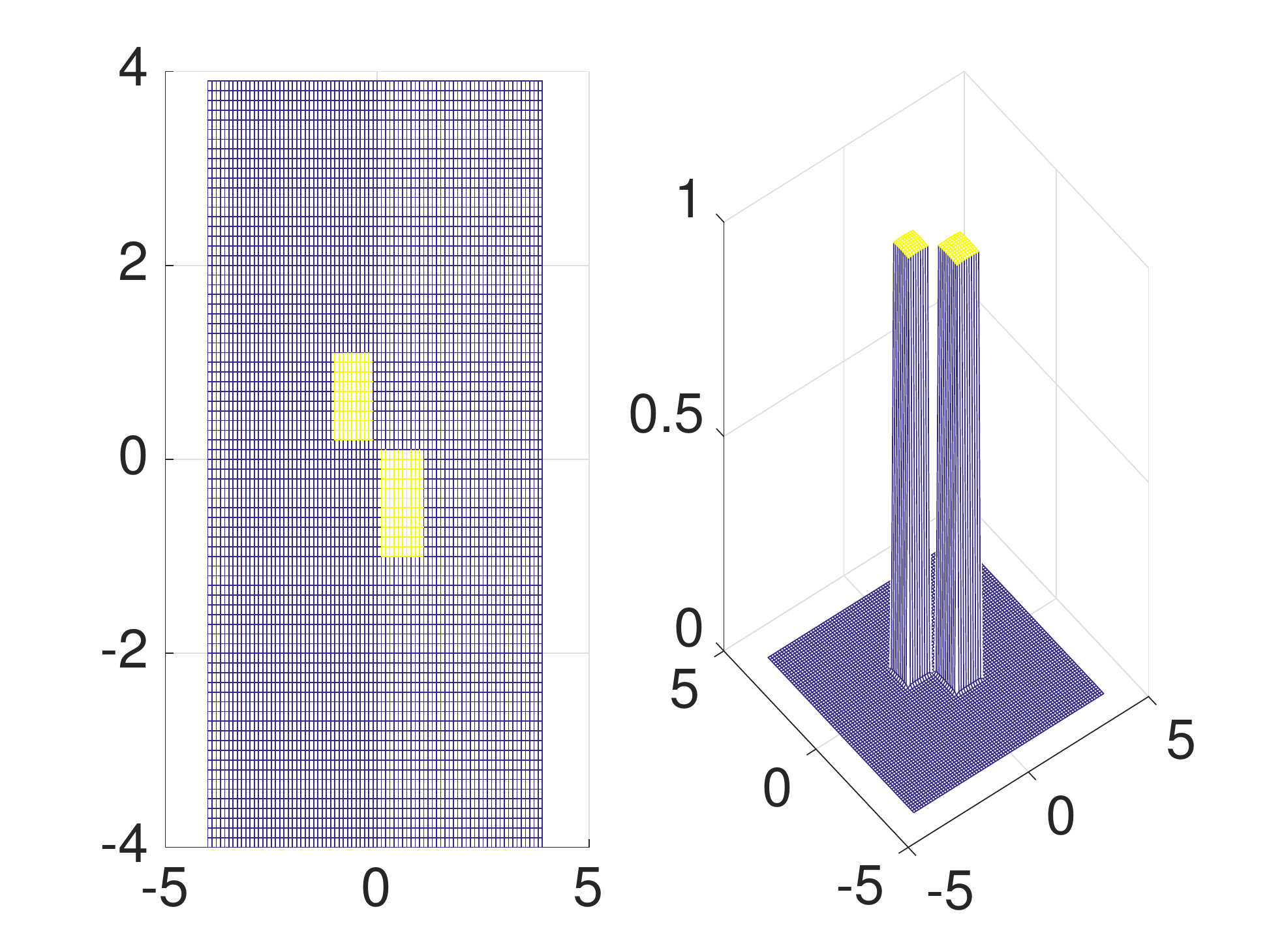}
\includegraphics[width = 0.5\textwidth]{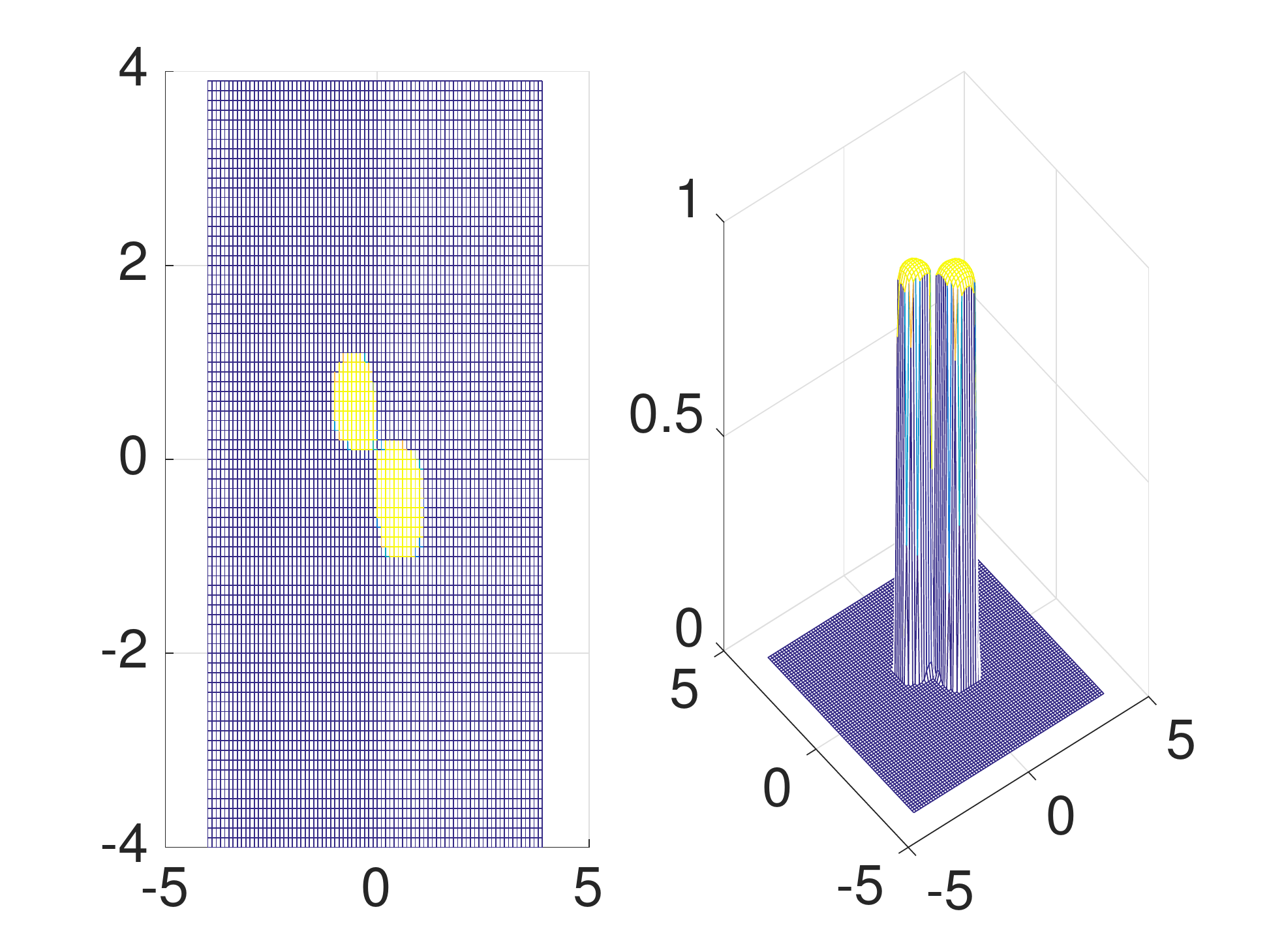}
\\ 
\includegraphics[width = 0.5\textwidth]{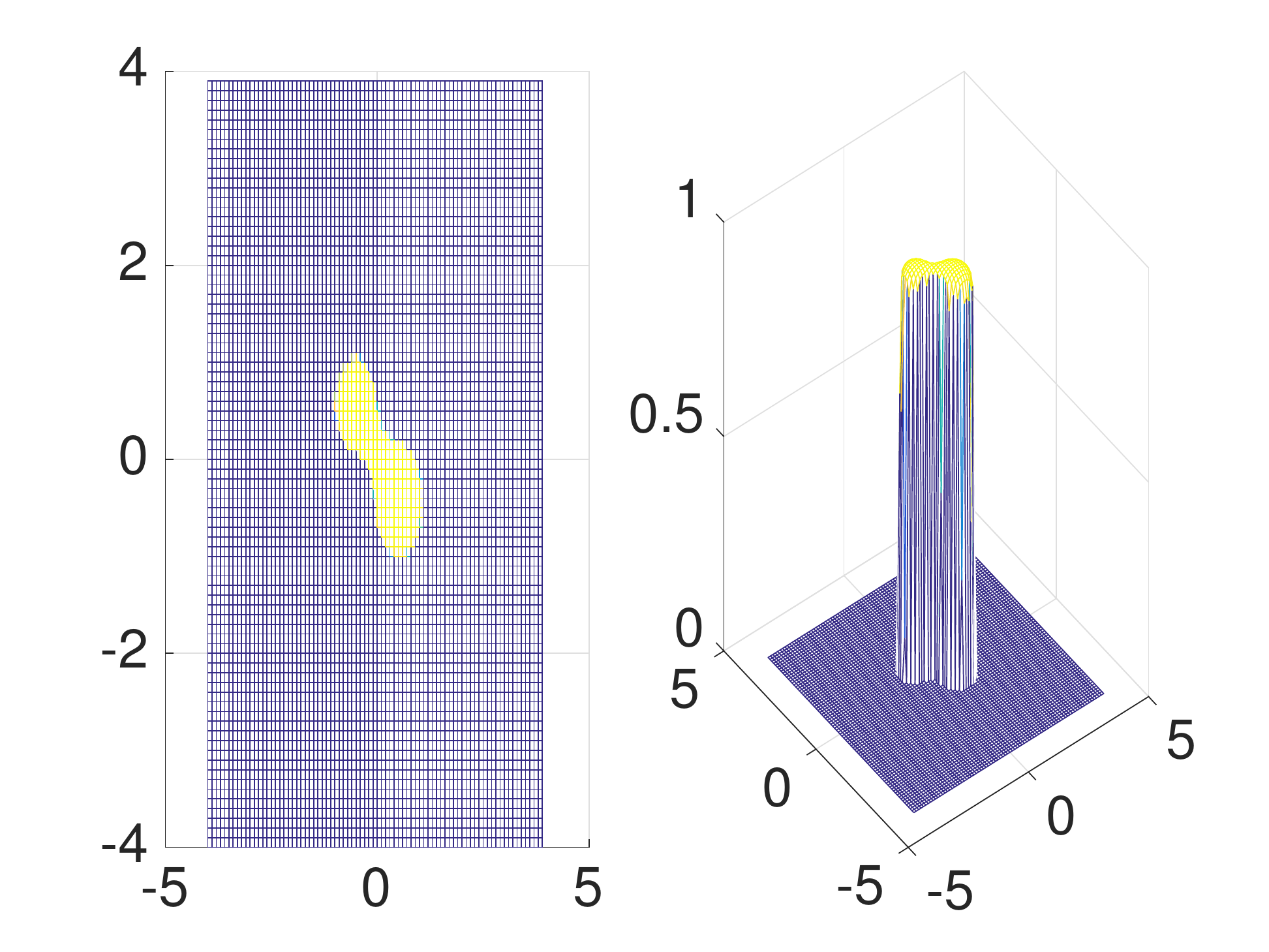}
\includegraphics[width = 0.5\textwidth]{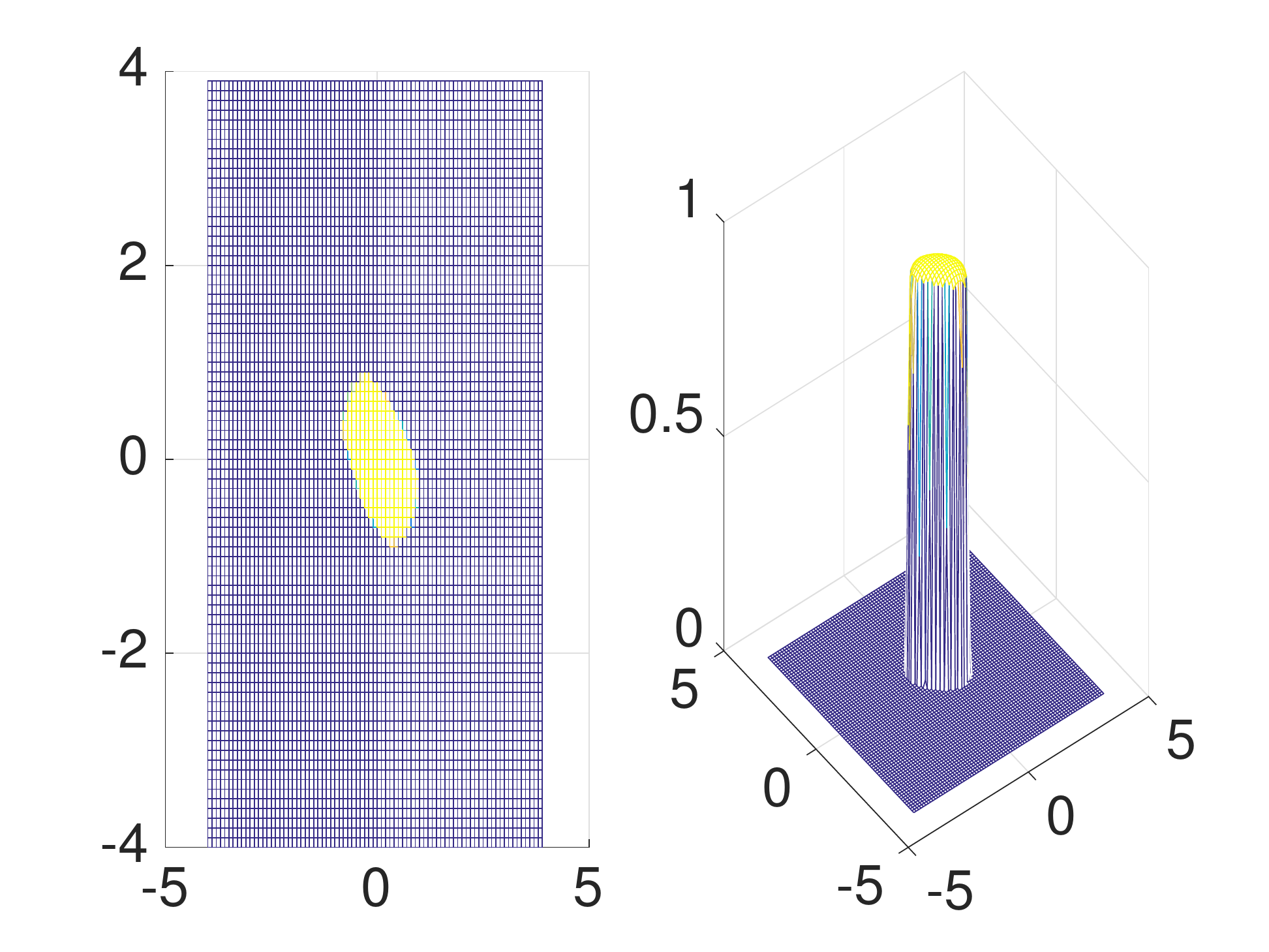}
\caption{Time evolution of model (\ref{eq:rhomn}) (\ref{eq:cmn}) with initial data (\ref{ic2}) (\ref{ic3}). Here $\eps = 0$, $m = 64$. Output times are: $t=0$ (upper left), $t=2$ (upper right), $t = 4$ (lower left) and $t=10$ (lower right). }
\label{fig:porous_twobump}
\end{figure}

\subsection{Two species}
In this section, we test our scheme on a two-species model \cite{KurganovLM2014}:
\begin{align}
\left\{ \begin{array}{cc} \partial_t \rho_1 + \chi_1 \nabla \cdot (\rho_1 \nabla c) = \mu_1 \Delta \rho_1 ,& 
\\  \partial_t \rho_2 + \chi_2 \nabla \cdot (\rho_2 \nabla c) = \mu_2 \Delta \rho_2, &  
\\  \eps c_t = D \Delta c + \alpha_1 \rho_1 + \alpha_2 \rho_2 - \beta c .& 
\end{array} \right. 
\end{align}
Here $\rho_1$ and $\rho_2$ denote the cell densities of the first and second species. $c$ is the concentration of the chemoattractant. $\mu_i$, $\chi_i$, $\alpha_i$ $i = 1, \ 2$, $\beta$, and $D$ are positive constants characterizing the cell diffusion, chemotactic sensitivities, production and consumption rates, and chemoattractant diffusion coefficient, respectively. A different combination of $\chi_1$, $\chi_2$ and the total mass of $\rho_1$ and $\rho_2$ would generate solutions with completely different behavior. Here we test our schemes in two specific combinations \cite{KurganovLM2014}, and other choices can be easily adapted and we omit the result here for simplicity. For both examples, we let $ \mu_2 = \gamma_1 = \gamma_2= \alpha_1 = \alpha_2 = D = 1$, and choose the computational domain to be $[-3,3]\times[-3,3]$.

{\bf Example 1:} First we choose $\chi_1 = 1$, $\chi_2 = 10$, $\mu_1 = 1 $, and initial condition is
\begin{equation}\label{ic5}
\rho_1(x,y,0) = \rho_2(x,y,0) = 50e^{-100(x^2 + y^2)}.
\end{equation}
In this case, we should have global existence in both $\rho_1$ and $\rho_2$. In Fig.~\ref{fig:2species1}, we plot $\rho_1$, $\rho_2$ and $c$ at $t=0.05$, and none of them displays any intensity of blowing up, yet $\rho_2$ has a sharper profile than $\rho_1$ since it has a large chemotactic sensitivity. 
\begin{figure}[!h]
\includegraphics[width = 0.32\textwidth]{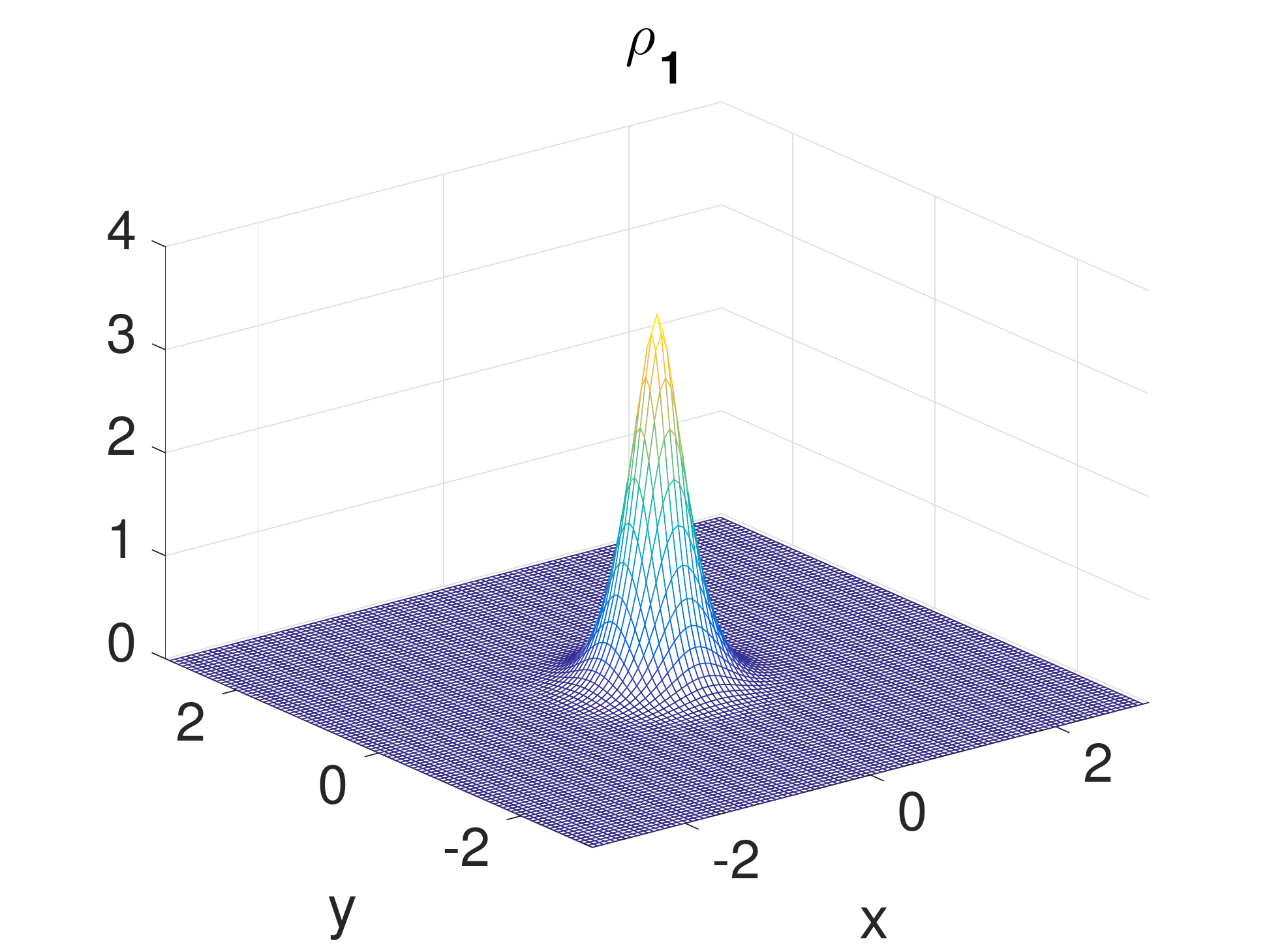}
\includegraphics[width = 0.32\textwidth]{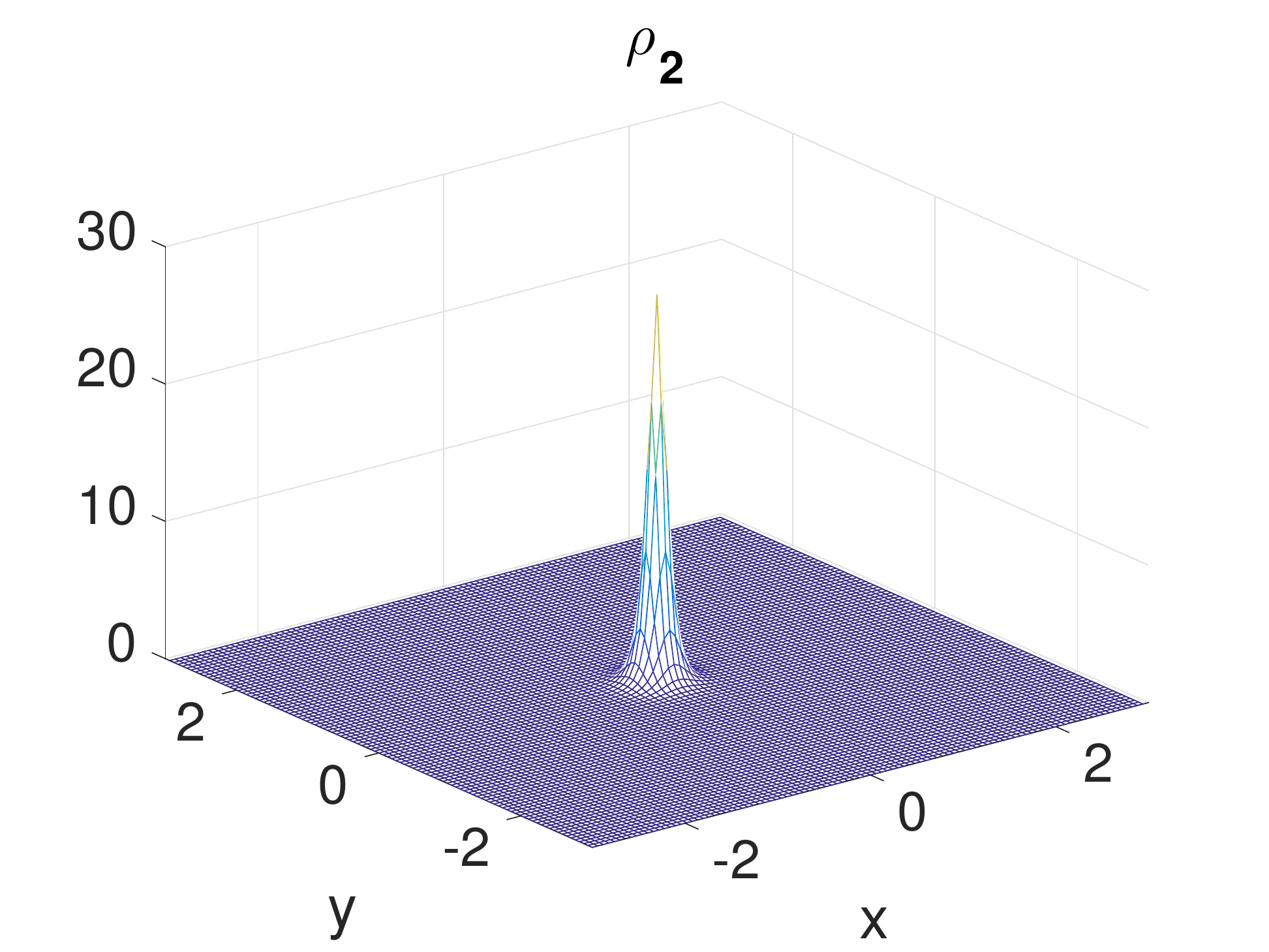}
\includegraphics[width = 0.32\textwidth]{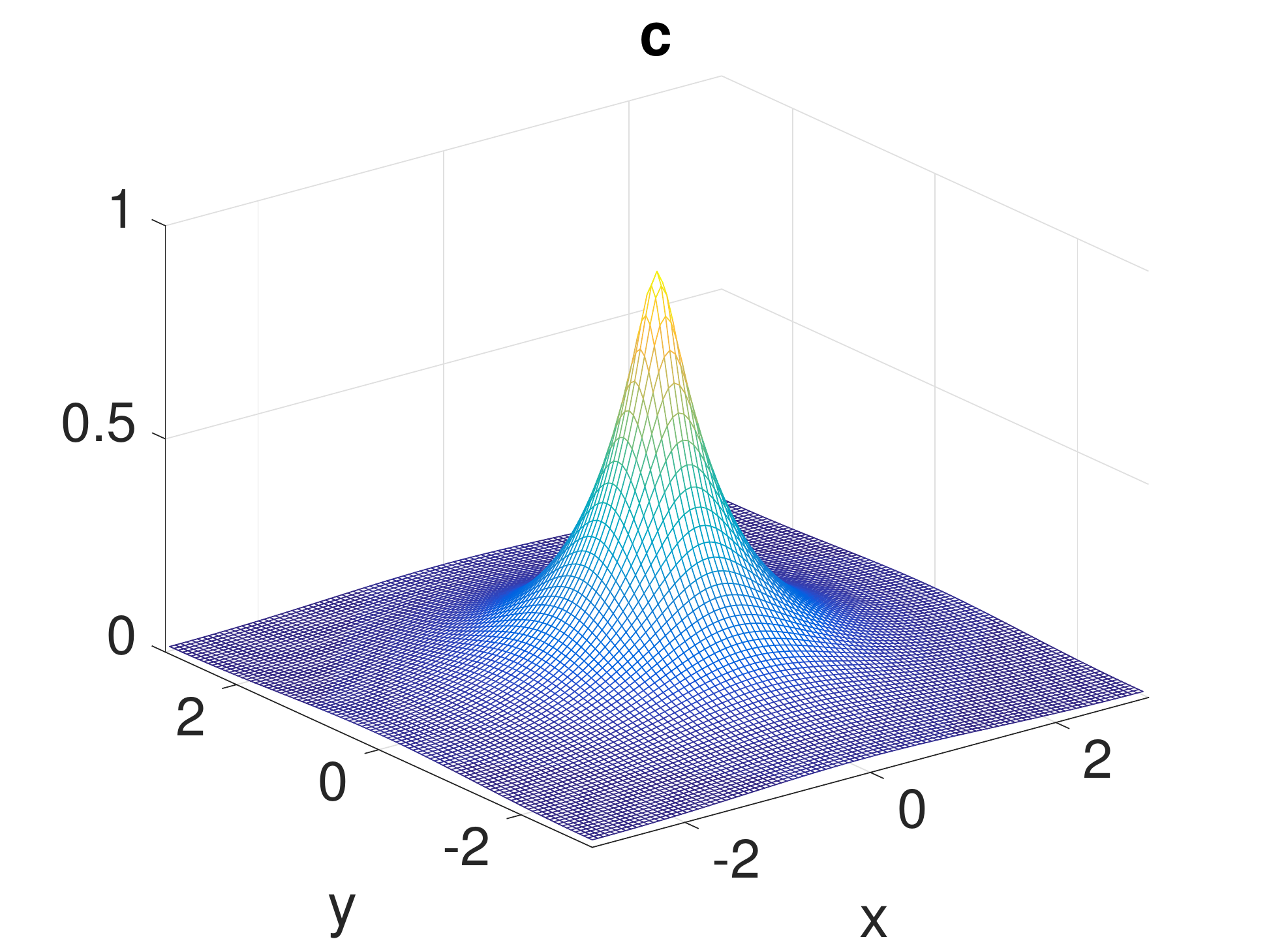}
\caption{Two sepeices: example1. $\rho_1$, $\rho_2$ and $c$ at time $t=0.05$, computed on $100\times 100$ uniform mesh. $\Delta t = \Delta x/10$.}
\label{fig:2species1}
\end{figure}

{\bf Example 2:} Next we consider $\chi_1 = 1$, $\chi_2 = 20$, $\mu_1 = 1 $, and use the same initial condition as in \eqref{ic5}. Here the problem falls into a subtle regime in which, according to \cite{EVC13}, should blow up $\rho_1$ and $\rho_2$ at different rate. Here we examine the profile of $\rho_1$ and $\rho_2$ at time $t = 0.05$ with two different mesh sizes, and it is seem from Fig.~\ref{fig:2species2} that both densities blow up at the order of $\mathcal{O}\left( \frac{1}{\Delta x^2}\right)$, but $\rho_2$ blow up faster than $\rho_1$.
\begin{figure}[!h]
\includegraphics[width = 0.5\textwidth]{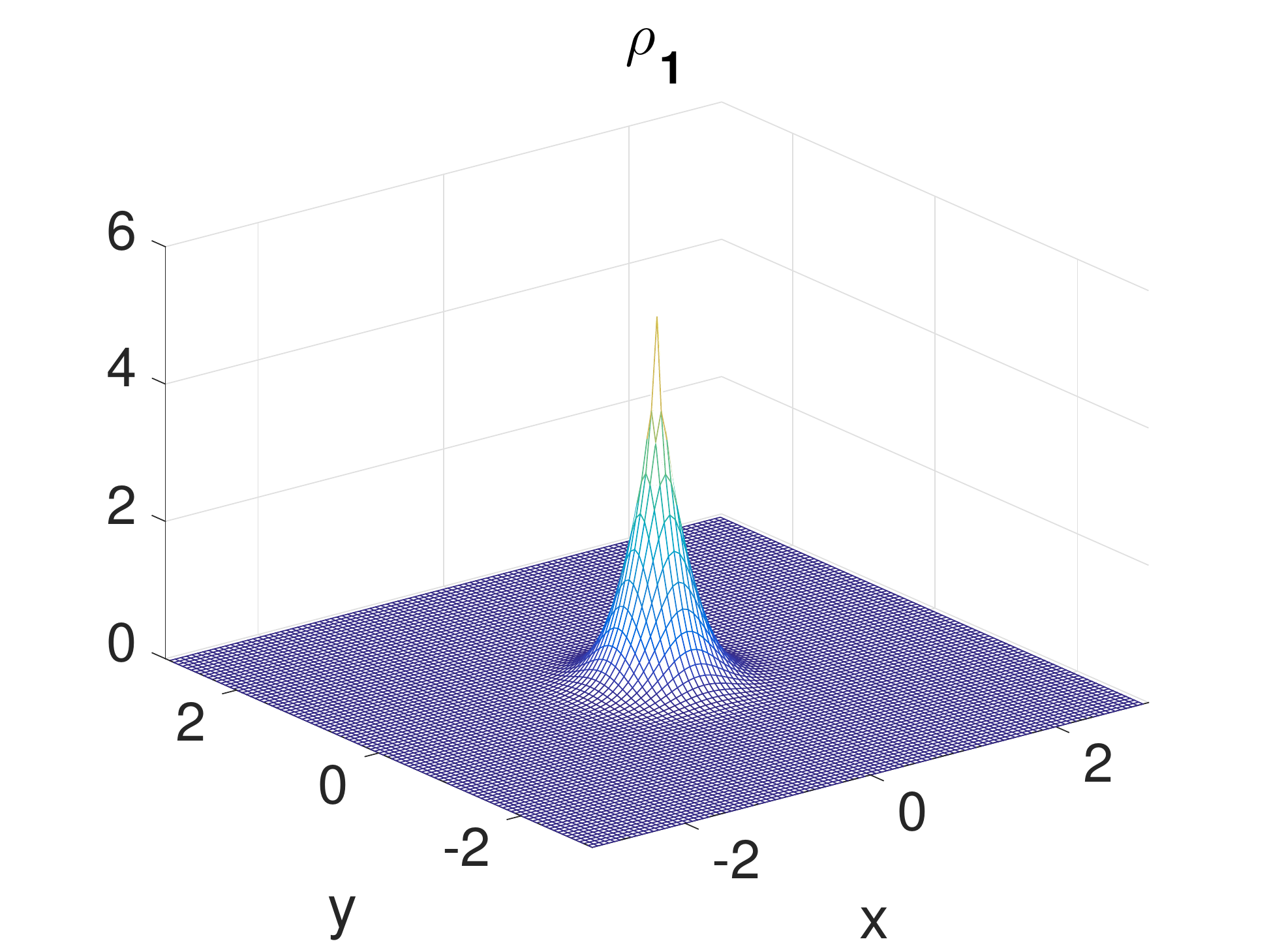}
\includegraphics[width = 0.5\textwidth]{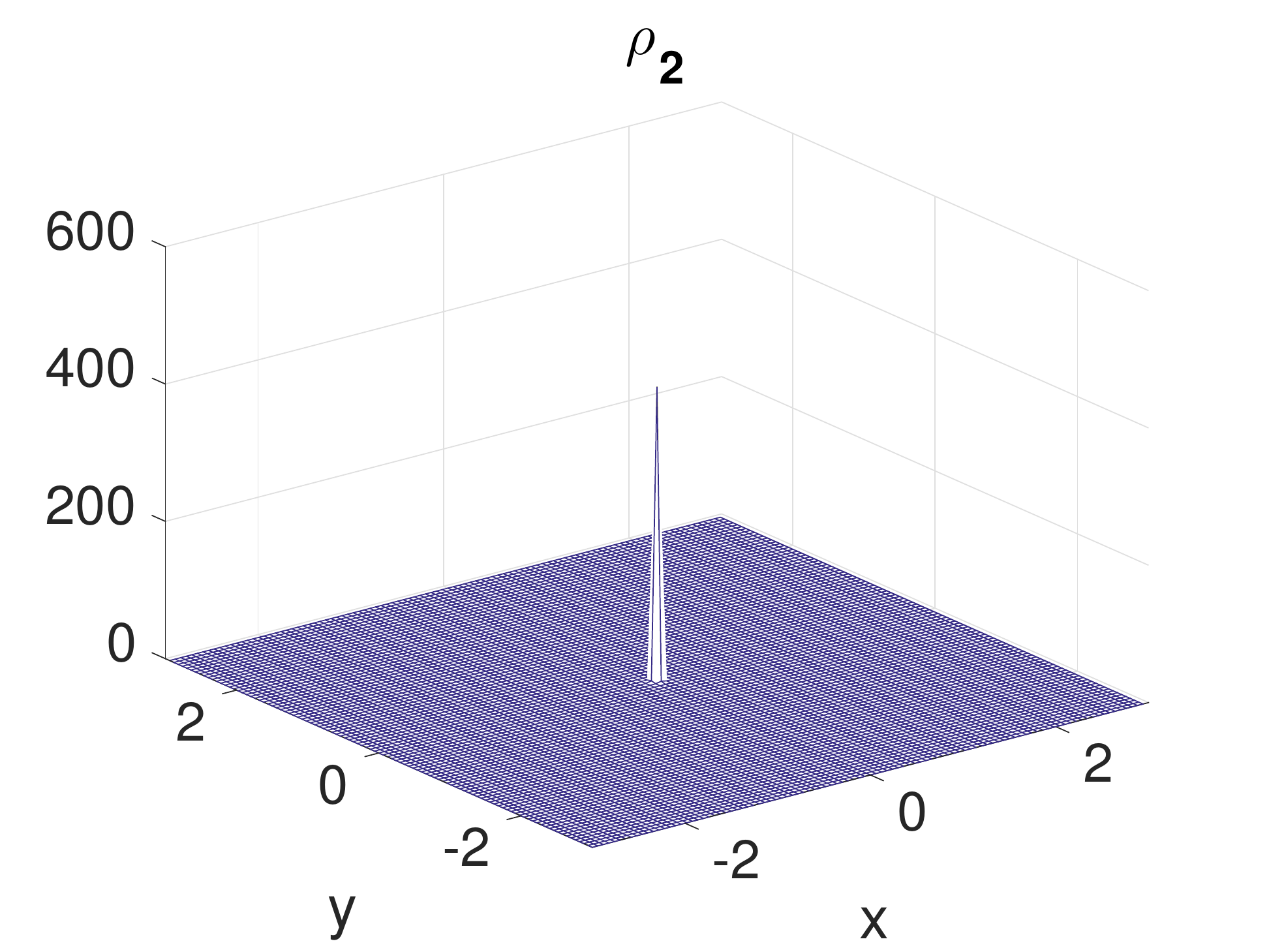}
\\
\includegraphics[width = 0.5\textwidth]{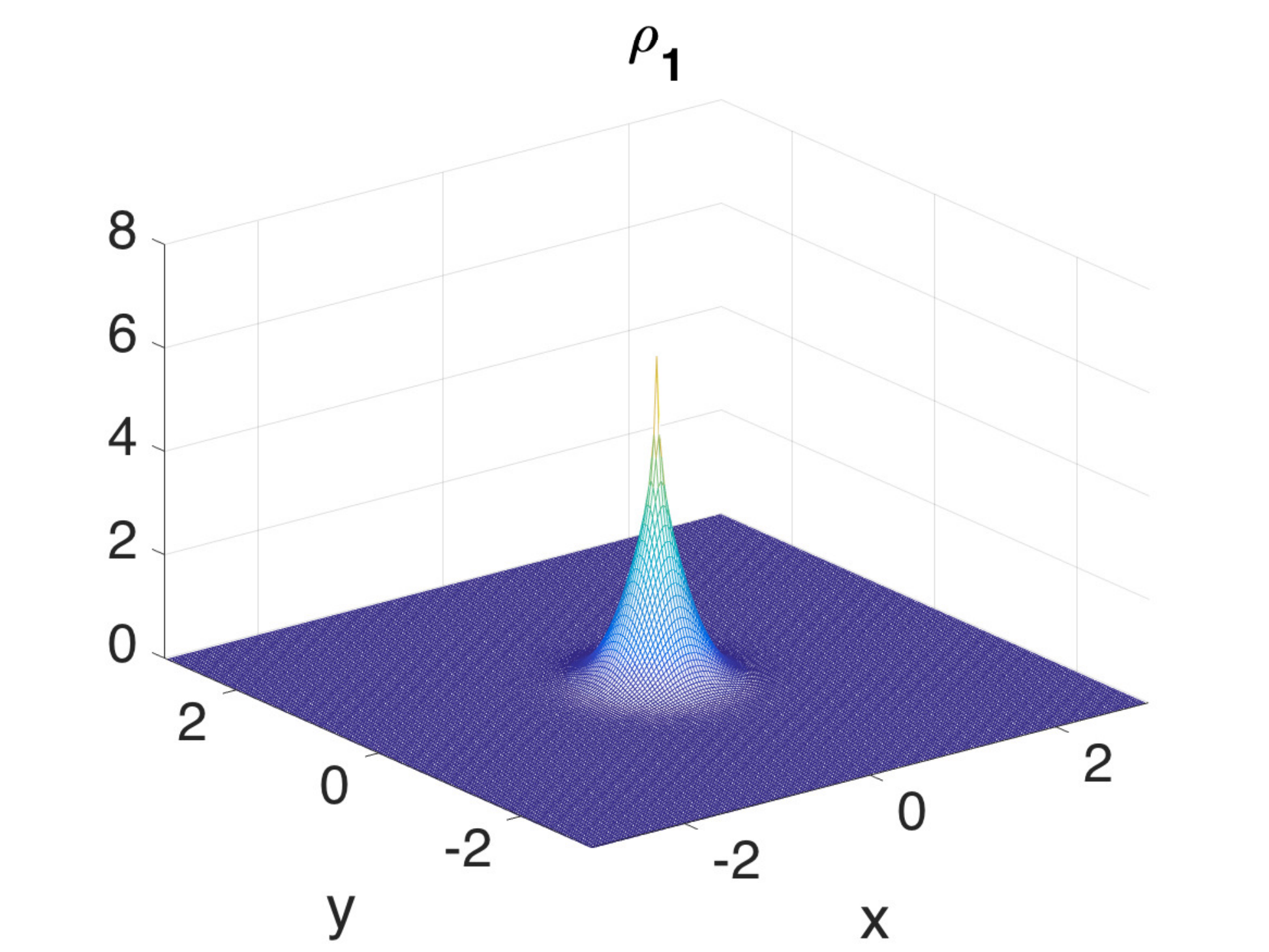}
\includegraphics[width = 0.5\textwidth]{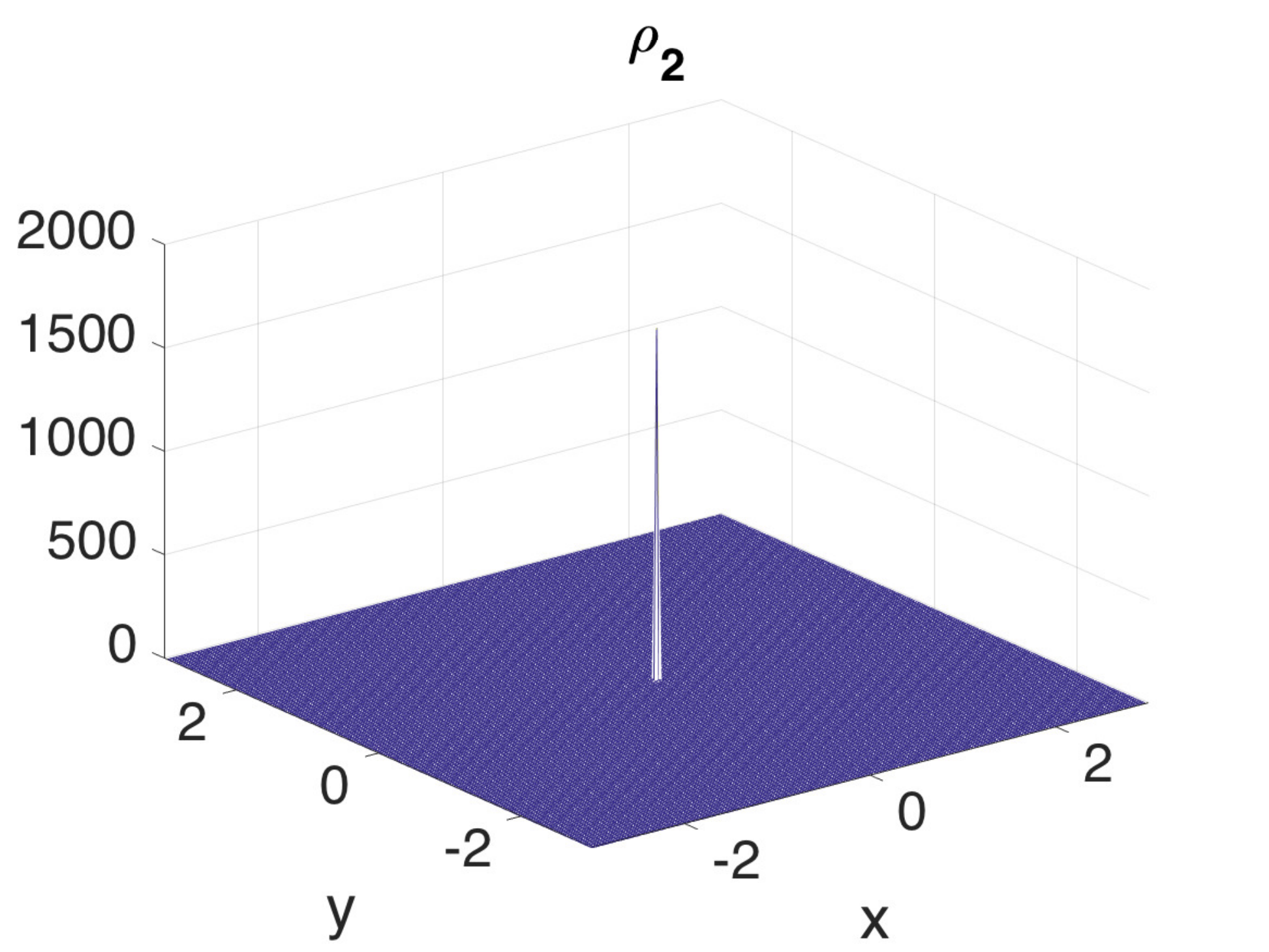}
\caption{Two sepeices: example2. $\rho_1$ and $\rho_2$ at time $t=0.05$, computed on $100\times 100$ uniform mesh (upper) and $200\times 200$ mesh (lower). $\Delta t = \Delta x/10$.}
\label{fig:2species2}
\end{figure}

\section*{Acknowledgments}
 The authors would like to thank Jianfeng Lu, Francis Filbet and Yao Yao for helpful discussions. J. Liu is partially supported by KI-Net NSF
RNMS grant No. 1107291 and NSF grant DMS 1514826. L. Wang is partially supported by NSF grant DMS 1620135. Z. Zhou is partially supported by RNMS11-07444 (KI-Net).


\end{document}